\documentclass[12pt]{article}
\usepackage[top=2cm, bottom=2cm, left=3cm , right=3cm]{geometry} 
\usepackage[utf8]{inputenc}
\usepackage[T1]{fontenc}
\usepackage{amssymb}
\usepackage{amsmath}
\usepackage{dsfont}
\usepackage{mathtools}
\usepackage{amsfonts}
\usepackage[greek, english]{babel}
\usepackage{amsthm}
\usepackage{tikz}
\usepackage{hyperref}
\usepackage{enumitem}  
\usepackage{marginnote}

\usepackage[draft,inline,nomargin,index]{fixme}

\fxsetup{theme=color,mode=multiuser}
\FXRegisterAuthor{cj}{acj}{\color{green}CJ}
\FXRegisterAuthor{mj}{amj}{\color{red}MJ}

\usepackage{setspace}
\addto\captionsfrench{%
  }

\newenvironment{cproof}{\begin{proof}[Proof of the 
		claim]}{\end{proof}}
\newtheoremstyle{capitale}    
    {\topsep}                    
    {\topsep}                    
    {\itshape}                   
    {}                           
    {\scshape}                   
    {. ---}                          
    {.5em}                       
    {} 
    
\newtheoremstyle{remark}    
    {\topsep}                    
    {\topsep}                    
    {}                   
    {}                           
    {\scshape}                   
    {. ---}                          
    {.5em}                       
    {} 

\theoremstyle{capitale}
\newtheorem{theorem}{Theorem}[section]
\newtheorem{lemma}[theorem]{Lemma}
\newtheorem*{claim*}{Claim}

\newtheorem*{lemme*}{Lemme}
\newtheorem{corollary}[theorem]{Corollary}
\newtheorem{definition}[theorem]{Definition}
\newtheorem*{definition*}{Definition}

\theoremstyle{remark}
\newtheorem{remark}[theorem]{Remark}
\newtheorem*{convention}{Convention}
\newtheorem{example}[theorem]{Example}
\newtheorem{question}[theorem]{Question}

\usepackage[backend=biber, style=numeric, maxnames=50]{biblatex}
\usepackage{csquotes}
\usepackage{todonotes}
\DefineBibliographyStrings{english}{in = {}}

\newcommand{\RR}{\mathcal{R}}

\newcommand{\PP}{\mathcal{P}}
\newcommand{\FF}{\mathcal{F}}

\newcommand{\LL}{\mathcal{L}}

\newcommand{\N}{\mathbb{N}}
\newcommand{\Z}{\mathbb{Z}}
\newcommand{\Q}{\mathbb{Q}}
\newcommand{\R}{\mathbb{R}}

\newcommand{\M}{\mathbf{M}}

\newcommand{\Sym}{\mathrm{Sym}}
\newcommand{\Struc}{\mathrm{Struc}}
\newcommand{\inv}{^{-1}}

\newcommand{\defin}[1]{\textbf{\textit{#1}}}
\newcommand{\actson}{\curvearrowright}
\newcommand{\Stab}{\mathrm{Stab}}

\newcommand{\IRS}{\mathrm{IRS}}

\newcommand{\Aut}{\mathrm{Aut}}
\newcommand{\Sub}{\mathrm{Sub}}
\newcommand{\Fix}{\mathrm{Fix}}

\newcommand{\tp}{\mathrm{tp}}

\renewcommand{\overline}{\bar}
\newcommand{\x}{\bar{x}}
\newcommand{\y}{\bar{y}}
\newcommand{\z}{\bar{z}}

\newcommand{\Addresses}{{
			\bigskip
			\footnotesize
			
			\noindent C.~Jahel, \textsc{Institut fur Algebra, Technische Universität Dresden, GERMANY}\par\nopagebreak\noindent
			\textit{E-mail address: }\texttt{colin.jahel@tu-dresden.de}
			
			\medskip
			
			\noindent M.~Joseph, \textsc{Université Paris-Saclay, CNRS, Laboratoire de mathématiques d’Orsay, 91405, Orsay, France}\par\nopagebreak\noindent
			\textit{E-mail address: }\texttt{matthieu.joseph@universite-paris-saclay.fr}
	
			\medskip

	}}

\title{Stabilizers for ergodic actions and invariant random expansions of non-archimedean Polish groups}
\author{Colin JAHEL \& Matthieu JOSEPH}
\date{}

\begin{document}
\maketitle

\abstract{Let $G$ be a closed permutation group on a countably infinite set $\Omega$, which acts transitively but not highly transitively. If $G$ is oligomorphic, has no algebraicity and weakly eliminates imaginaries, we prove that any probability measure preserving ergodic action $G\curvearrowright (X,\mu)$ is either essentially free or essentially transitive. As this stabilizers rigidity result concerns a class of non locally compact Polish groups, our methods of proof drastically differ from that of similar results in the realm of locally compact groups. We bring the notion of dissociation from exchangeability theory in the context of stabilizers rigidity by proving that if $G\lneq\Sym(\Omega)$ is a transitive, proper, closed subgroup, which has no algebraicity and weakly eliminates imaginaries, then any dissociated probability measure preserving action of $G$ is either essentially free or essentially transitive. A key notion that we develop in our approach is that of invariant random expansions, which are $G$-invariant probability measures on the space of expansions of the canonical (model theoretic) structure associated with $G$. We also initiate the study of invariant random subgroups for Polish groups and prove that -- although the result for p.m.p.\ ergodic actions fails for the group $\Sym(\Omega)$ of all permutations of $\Omega$ -- any ergodic invariant random subgroup of $\Sym(\Omega)$ is essentially transitive.}
\\

\noindent
\textbf{MSC:}  Primary: 37A15, 22F50. Secondary: 03C15, 03C98, 60G09, 03C75.\\
\textbf{Keywords: }Measure-preserving actions, Polish groups, non-archimedian groups, exchangeability, invariant random subgroups, model theory, infinitary logic.
\tableofcontents

\section{Introduction}

Ergodic theory of group actions has a longstanding history whose foundations started nearly a century ago with the study of actions of the group of integers thanks to the notion of measure-preserving transformation, and of the group of real numbers thanks to the notion of flow. It is now a well established theory in the context of locally compact groups, especially since the groundbreaking works of Zimmer on ergodic theory for semisimple Lie groups. A principal object of study in ergodic theory is that of probability measure-preserving actions (\emph{abbrv.}\ p.m.p.\ actions). A p.m.p.\ action $G\curvearrowright (X,\mu)$ of a topological group $G$ is a Borel action $G\curvearrowright X$ on a standard Borel space $X$ (that is, the action map $(g,x)\mapsto g\cdot x$ is Borel measurable) with a $G$-invariant Borel probability measure $\mu$ on $X$, i.e.\ satisfying $\mu(g\cdot Y)=\mu(Y)$ for all measurable subset $Y\subseteq X$ and all $g\in G$. The action is ergodic if any measurable subset $Y\subseteq X$ which is $G$-invariant in the sense that $\mu(Y\triangle g\cdot Y)=0$ for all $g\in G$, satisfies $\mu(Y)\in\{0,1\}$. Classically, ergodic theory restricts to p.m.p.\ actions of locally compact (second countable) groups: firstly, any such group $G$ admits non-trivial p.m.p.\ actions, such as Gaussian actions or Poisson point processes. Moreover, orbits of p.m.p.\ actions of any such $G$ have a nice structure: they are homogeneous spaces carrying a unique invariant measure class, induced by the Haar measure on $G$. 

Striking rigidity results have been proved in the setting of ergodic theory of locally compact groups. One of the most remarkable is the Stuck-Zimmer stabilizer rigidity theorem  \cite{StuckZimmer}, which is a (far reaching) generalization of Margulis normal subgroup theorem. It asserts that \emph{any p.m.p.\ ergodic action  of a higher rank simple Lie group (e.g.\ $\mathrm{SL}_n(\R)$ for $n\geq 3$) is either \defin{essentially free} (a conull set of points have a trivial stabilizer) or \defin{essentially transitive} (one orbit has full measure)}.

In the present paper, we study ergodic theory of Polish groups, which are separable, completely metrizable, topological groups. Our principal aim is to develop a framework to obtain stabilizer rigidity results à la Stuck-Zimmer for a large class of non locally compact Polish groups. The general ergodic theory for Polish groups is nonetheless much more pathological than that of locally compact groups, as reflected by the existence of Polish groups which admit no non-trivial p.m.p.\ action \cite{GlasnerTsirelsonWeiss}, \cite{GlasnerWeissSpatial}. 

This motivates our choice to restrict our attention to a specific class of Polish groups: that of non-archimedean Polish groups. A non-archimedean Polish group is a Polish group which admits a basis at the identity consisting of open subgroups. Non-archimedean Polish group can be characterized in several ways \cite[\S~1.5]{BeckerKechris}. They are automorphism groups of countable structures, but can also be defined as closed permutation groups, that is, closed subgroups of the symmetric group $\Sym(\Omega)$ (the group of \emph{all} permutations) on a countably infinite set $\Omega$. Here $\Sym(\Omega)$, and therefore any closed permutation group, has the topology of pointwise convergence which turns it into a Polish group. With this last characterisation in mind, one notices that non-archimedean Polish groups have plenty of non-trivial p.m.p.\ actions, such as generalized Bernoulli shift actions on product probability spaces of the form $(A,\kappa)^{\Omega}$ or $(A,\kappa)^{\Omega^n}$. 

Ergodic theoretic aspects of p.m.p.\ actions of the most iconic non-archimedean Polish group, namely $\mathrm{Sym}(\Omega)$, have been studied from different contexts (exchangeability theory, model theory, etc.). In fact, $\Sym(\Omega)$ admits p.m.p.\ ergodic actions that are neither essentially free, nor essentially transitive, see Remark \ref{rem.Sinftynonessfreenonesstrans}. Such p.m.p.\ actions for $\Sym(\Omega)$ have been studied in-depth from a model theoretic perspective in \cite{AFKP}. As it turns out, there are plenty of non locally compact closed permutation groups which satisfy a stabilizer rigidity result à la Stuck-Zimmer. 

\begin{theorem}\label{thmintro.pmprigide}
    Let $G\lneq\Sym(\Omega)$ be a transitive, proper, closed subgroup. Assume that 
    \begin{itemize}
        \item (oligomorphy) the diagonal action $G\curvearrowright\Omega^n$ has only finitely many orbits for each integer $n\geq 1$, 
        \item (no algebraicity) the pointwise stabilizer $G_A\coloneqq\{g\in G\colon \forall a\in A, g(a)=a\}$ acts on $\Omega\setminus A$ without fixed point for each finite subset $A\subseteq\Omega$,
        \item (weak elimination of imaginaries) every open subgroup of $G$ contains, as a finite index subgroup, the pointwise stabilizer $G_A$ of a finite subset $A\subseteq\Omega$. 
    \end{itemize}
    Then any p.m.p.\ ergodic action $G\curvearrowright (X,\mu)$ is either essentially free or essentially transitive.
\end{theorem}

Note that by Lemma \ref{lem.oligomorphicnotloccompact} closed permutation groups which are oligomorphic groups are not locally compact. Theorem \ref{thmintro.pmprigide} applies to a large variety of non locally compact Polish groups (in fact continuum many) such as the group \mbox{$\Aut(\Q,<)$} of order-preserving bijections of $\Q$, the group $\Aut(\Q/\Z,<)$ of bijections of $\Q/\Z$ which preserve the dense cyclic order, the automorphism group of the Rado graph and many more. We refer to Example \ref{ex.DdFgroups} for a wider variety of groups covered by our theorem. 

The assumptions on the groups in Theorem \ref{thmintro.pmprigide} are rather standard. In fact, the class of closed permutation groups $G\leq\Sym(\Omega)$, which are oligomorphic, have no algebraicity and admit weak elimination of imaginaries, has been studied in various contexts, see for instance \cite{JT} or \cite{Tsankov}. 

\begin{remark}
    For locally compact groups, stabilizers rigidity results have flourished since the groundbreaking theorem of Stuck and Zimmer \cite{StuckZimmer}, see for instance \cite{BaderShalom}, \cite{Creutz},  \cite{CreutzPeterson}, \cite{HartmanTamuz} and \cite{Levit}. Despite the similarities in statement, our proof requires some markedly different methods, as the groups involved in Theorem \ref{thmintro.pmprigide} are non locally compact Polish groups. 
\end{remark}

\begin{remark}
   Closed permutation groups $G\leq\Sym(\Omega)$ have no shortage of p.m.p.\ actions. Indeed, for any standard probability space $(A,\kappa)$, the generalized Bernoulli shift $G\curvearrowright (A,\kappa)^\Omega$ defined by $(g,(a_\omega)_{\omega\in\Omega})\mapsto (a_{g\inv(\omega)})_{\omega\in\Omega}$ is a p.m.p.\ action. We will prove in Lemma \ref{lem.generalizedbernoulli} that if $G$ is proper, has no algebraicity and weakly eliminates imaginaries, this action is essentially transitive when $(A,\kappa)$ is purely atomic and essentially free otherwise. 
\end{remark}

Given a p.m.p.\ ergodic action $G\curvearrowright(X,\mu)$ and a finite subset $A\subseteq\Omega$, we denote by $\FF_A$ the $\sigma$-algebra of measurable subsets that are $G_A$-invariant. The p.m.p.\ action $G\curvearrowright(X,\mu)$ is \defin{dissociated} if for all finite disjoint subsets $A,B\subseteq\Omega$, the $\sigma$-algebras $\FF_A$ and $\FF_B$ are independent. A key result in our study is that any p.m.p.\ ergodic action of any closed permutation group $G\leq \Sym(\Omega)$ which is oligomorphic, has no algebraicity and admits weak elimination of imaginaries, is dissociated, see \cite[Thm.~3.4]{JT}. We prove the following general result, which implies Theorem \ref{thmintro.pmprigide}.

\begin{theorem}\label{thmintro.pmprigiddissociated}
    Let $G\lneq\Sym(\Omega)$ be a transitive, proper, closed subgroup. If $G$ has no algebraicity and is primitive, then any p.m.p.\ ergodic action $G\curvearrowright(X,\mu)$ which is dissociated is either essentially free or essentially transitive. 
\end{theorem}

Recall that $G\leq \Sym(\Omega)$ is primitive if there are no $G$-invariant equivalence relations on $\Omega$ apart from equality and $\Omega\times\Omega$. Equivalently, $G\leq\Sym(\Omega)$ is primitive if $G_a$ is a maximal subgroup of $G$ for all $a\in\Omega$. As the assumptions ``weak elimination of imaginaries and no algebraicity'' imply primitivity (see Corollary \ref{cor.WEI+noAlgimpliesprimitive}), Theorem \ref{thmintro.pmprigiddissociated} implies Theorem \ref{thmintro.pmprigide}. In fact, we prove a more precise result as we are able to show that the dichotomy essentially free/essentially transitive for a p.m.p.\ action $G\curvearrowright (X,\mu)$ is characterized by the existence or absence of a fixed point for the action $\Stab(x)\curvearrowright\Omega$, where $x\in X$ and  $\Stab(x)\coloneqq\{g\in G\colon g\cdot x=x\}$.

\begin{theorem}[version for p.m.p.\ actions, see Theorem \ref{thm.pmprigiddissociatedexpanded}]\label{thmintro.pmprigiddissociatedexpanded}
    Let $G\lneq \Sym(\Omega)$ be a transitive, proper, closed subgroup. If $G$ has no algebraicity and is primitive, then for any dissociated p.m.p.\ action $G\curvearrowright(X,\mu)$, the following holds:
    \begin{itemize}
        \item either $\Stab(x)\curvearrowright\Omega$ has a fixed point for $\mu$-a.e.\ $x\in X$ and in this case $G\curvearrowright(X,\mu)$ essentially free,
        \item or $\Stab(x)\curvearrowright\Omega$ has no fixed point for $\mu$-a.e.\ $x\in X$ and in this case $G\curvearrowright(X,\mu)$ is essentially transitive. 
    \end{itemize}
\end{theorem}
Theorem \ref{thmintro.pmprigiddissociated} is now a straightforward consequence of Theorem \ref{thmintro.pmprigiddissociatedexpanded}. Let us emphasize that the assumption ``$G$ is a proper subgroup of $\Sym(\Omega)$'' can be dropped in the second item of the above theorem (see Theorem \ref{thm:NoFixIsConcOnOrbitpmp}) and this will be useful to prove that any ergodic invariant random subgroup of $\Sym(\Omega)$ is concentrated on a conjugacy class. However, as we already discussed, this assumption is essential in the first item of the above theorem. We will see in Remark \ref{rem.Sinftynonessfreenonesstrans} that $\Sym(\Omega)$ admits dissociated p.m.p.\ actions $G\curvearrowright(X,\mu)$ that are not essentially free and for which $\Stab(x)\curvearrowright\Omega$ has a fixed point for $\mu$-a.e.\ $x\in X$. \\

One of the behaviors that we exhibit -- essential transitivity of p.m.p.\ actions -- can be considered as a measure theoretic equivalent to the topological property for a minimal action of having a comeager orbit. Topological dynamics of closed subgroups of $\Sym(\Omega)$ is an extensively studied topic that was kindled by Kechris, Pestov and Todorčević in \cite{KPT}. An especially important result in this area, due Ben-Yaacov, Melleray, Nguyen Van Thé, Tsankov and Zucker in \cite{BMT}, \cite{MNT} and \cite{Zucker} is the classification of groups for which every minimal action on a compact Hausdorff space admits a comeager orbit. In particular, they show a link between the existence of those comeager orbits and the metrizability of the universal minimal flow, both phenomena corresponding to the existence of a suitable expansion of the canonical structure associated with the group. This mirrors our own study, as we explore invariant random expansions of structures in order to study p.m.p.\ actions.

The point of view that we adapt in order to prove Theorem \ref{thmintro.pmprigiddissociatedexpanded} comes from model theory (basic model-theoretic notions will be discussed in Section \ref{sec.Prelim}). This is motivated by the fact that any closed permutation group $G$ is indeed (isomorphic to) the automorphism group of a countable relational structure (see \cite[\S~1.5]{BeckerKechris}), namely the canonical structure associated with $G$.  Given any countable relational language $\LL$ (which contains the canonical language $\LL_G$ associated with $G$), we denote by $\mathrm{Struc}_\LL^G$ the compact space of \textit{expansions} of $\M_G$ in the language $\LL$. These are structures whose reduct (the structure obtained by removing the relations in $\LL\setminus\LL_G$) is equal to $\M_G$. The space $\Struc_\LL^G$ carries a continuous $G$-action and we will study the $G$-invariant Borel probability measures for this action, that we call \defin{invariant random expansions} of (the canonical structure associated with) $G$. Invariant random expansions, IREs for short, of $\Sym(\N)$ were studied from the model theoretic point of view in a series of papers under different names such as invariant measures, invariant structures, or ergodic structures \cite{AFKP}, \cite{AFKwP}, \cite{AFNP}, \cite{AFPtheory}, \cite{AFP},  \cite{AFPentropy}.  We prefer here to coin the name IRE as it more accurately describes the objects, the structures we look at being explicitly expansions. Furthermore, this name is reminiscent to the acronym for invariant random subgroups (IRS), a topic that we will discuss in the context of Polish groups. Let us denote by $\mathrm{IRE}_\LL(G)$ the space of invariant random expansions of $G$ in the language $\LL$. Theorem \ref{thmintro.pmprigiddissociatedexpanded} can be translated in terms of invariant random expansions, as expressed below. This new version is in fact equivalent to Theorem \ref{thmintro.pmprigiddissociatedexpanded}, using a universality theorem due to Becker and Kechris \cite[Thm.~2.7.4]{BeckerKechris}, see Lemma \ref{lem.dictionnary} for a variation of the latter result that is adapted to our context.

\addtocounter{theorem}{-1}

\begin{theorem}[version for IREs]\label{thmintro.IRErigide}
    Let $G\lneq\Sym(\Omega)$ be a transitive, proper, closed subgroup. Let $\LL$ be a countable relational language which contains the canonical language $\LL_G$. If $G$ has no algebraicity and is primitive, then for any $\mu\in\mathrm{IRE}_\LL(G)$ dissociated, the following holds:
    \begin{itemize}
        \item either $\Aut(\M)\curvearrowright\Omega$ has a fixed point for $\mu$-a.e.\ $\M\in\Struc_\LL^G$ and in this case $G\curvearrowright(\Struc_\LL^G,\mu)$ is essentially free,
        \item or $\Aut(\M)\curvearrowright\Omega$ has no fixed point for $\mu$-a.e.\ $\M\in\Struc_\LL^G$ and in this case $G\curvearrowright(\Struc_\LL^G,\mu)$ is essentially transitive. 
    \end{itemize}
\end{theorem}
As we already discussed before, the assumption that ``$G$ is a proper subgroup of $\Sym(\Omega)$'' ca be dropped in the second item but is essential in the first one. \\

Another part of our work concerns subgroup dynamics for Polish groups and more precisely for closed permutation groups equipped with the pointwise convergence topology. \defin{Subgroup dynamics} is the study for a topological group $G$ of its action by conjugation on the set $\Sub(G)$ of its \emph{closed} subgroups. It turns out that subgroup dynamics is a very active area of research in the locally compact realm. In the presence of a Polish locally compact group $G$, the space of closed subgroups $\Sub(G)$ is endowed with a natural topology, called the Chabauty topology, which turns $\Sub(G)$ into a compact Hausdorff space. In this setting, subgroup dynamics turned out to be very fruitful in various contexts of group theory such as the study of lattices in Lie groups \cite{ABBGNRS}, $\mathrm{C}^*$-simplicity \cite{Kennedy}, or else permutation stability of groups \cite{BeckerLubotzkyThom}. Yet the picture is not quite as rosy in the Polish realm. If $G$ is a non locally compact Polish group, then the Chabauty topology on the space $\Sub(G)$ of closed subgroups is not Hausdorff in general (and this is indeed not the case when $G=\Sym(\Omega)$). However, there is a natural $\sigma$-algebra on $\Sub(G)$ called the \defin{Effros $\sigma$-algebra} that turns $\Sub(G)$ into a standard Borel space. This allows to define IRSs for Polish groups. 

An \defin{invariant random subgroup} (IRS) of a Polish group $G$ is a probability measure on (the Effros $\sigma$-algebra of) $\Sub(G)$ which is invariant by conjugation. We denote by
\[\mathrm{IRS}(G)\coloneqq \mathrm{Prob}(\Sub(G))^G\] the (standard Borel) space of invariant random subgroups of $G$. The theory of invariant random subgroups as well as the spaces $\mathrm{IRS}(G)$ in the setting of Polish \emph{locally compact} groups $G$ have been very recently extensively studied on their own and the literature in this area is rapidly growing, see e.g.\ \cite{AbertGlasnerVirag}, \cite{BaderDuchesneLecureuxWesolek}, \cite{Bowen}, \cite{BowenGrigorchukKravchenko}. On the contrary, invariant random subgroups of non locally compact Polish groups have not been studied so far to the best of our knowledge. 

One of the foundational results for IRSs of locally compact groups is that every $\mu\in\IRS(G)$ is obtained as the stabilizer IRS of a p.m.p.\ action $G\curvearrowright(X,\mu)$, i.e., the stabilizer of a $\mu$-random point \cite[Thm.~2.6]{ABBGNRS} (see also \cite[Prop.~12]{AbertGlasnerVirag} for a proof when $G$ is countable). We prove a similar statement for closed permutation groups.

\begin{theorem}[see Theorem \ref{thm.IRSrealised}]
    Let $G\leq\Sym(\Omega)$ be a closed subgroup and let $\nu\in\mathrm{IRS}(G)$. Then there exists a p.m.p.\ action $G\curvearrowright (X,\mu)$ whose stabilizer IRS is equal to $\nu$.
\end{theorem}

We say that $\nu\in\IRS(G)$ is \defin{concentrated on a conjugacy class} if there exists an orbit $O$ of the $G$-action by conjugation on $\Sub(G)$ such that $\nu(O)=1$. Theorem \ref{thmintro.pmprigide} readily implies that if $G\lneq\Sym(\Omega)$ is a transitive, proper, closed subgroup, which is oligomorphic, has no algebraicity and weakly eliminates imaginaries, then any ergodic $\nu\in\IRS(G)$ is concentrated on a conjugacy class. Even though Theorem \ref{thmintro.pmprigide} is false for $\Sym(\Omega)$, we prove that any ergodic IRS of $\Sym(\Omega)$ is concentrated on a conjugacy class, therefore obtaining the following result. 

\begin{theorem}[see Theorem \ref{thm.IRSrigidity}]
    Let $G\leq\Sym(\Omega)$ be a transitive closed subgroup. If $G$ is oligomorphic, has no algebraicity and weakly eliminates imaginaries, then any ergodic $\nu\in\IRS(G)$ is concentrated on a conjugacy class.
\end{theorem}

\begin{convention} 
In this paper, all countable relational structures as well as all closed permutation groups are defined on $\Omega=\N$. We denote by $S_\infty$ the group $\mathrm{Sym}(\N)$ of all permutations on $\N$. Tuples will be denoted by $\x,\y,\dots$ and $\N^{<\omega}$ will denote the set of tuples on $\N$.   
\end{convention}

\paragraph{Acknowledgments.} We would like to thank Gianluca Basso, Ronnie Chen, Clinton Conley, David Evans, François le Maître, Todor Tsankov and Anush Tserunyan for fruitful discussions related to this work. We also thank Nate Ackerman, Cameron Freer and Rehana Patel for comments on a draft of this paper. C.J. was partially funded by the Deutsche Forschungsgemeinschaft (DFG, German Research Foundation) – project number 467967530. M.J.\ was partially supported by a public grant as part of the Investissement d'avenir project, reference ANR-11-LABX-0056-LMH, LabEx LMH.

\section{Dissociated p.m.p.\ actions and oligomorphic groups}\label{sec.deFinetti}

We start by the straightforward following folklore result which shows that the groups under considerations in this paper are non locally compact Polish groups.

\begin{lemma}\label{lem.oligomorphicnotloccompact}
    Let $G\leq S_\infty$ be a closed subgroup. If $G$ is either oligomorphic or has no algebraicity, then $G$ is not locally compact.
\end{lemma}
\begin{proof}
    It suffices to prove that for all $A\subseteq\N$ finite, $G_A$ is not compact. Take $x\notin A$ with infinite orbit under $G_A$, which exists either by oligomorphy or since $G$ has no algebraicity. Let $(x_i)$ be an enumeration of the $G_A$-orbit of $x$. The open sets $G_A\cap \{g\in G \colon g(x)=x_i\}$ form an infinite open cover by disjoint sets of $G_A$, it is therefore not compact.
\end{proof}
For a p.m.p.\ action $G\curvearrowright (X,\mu)$ of a closed subgroup $G\leq S_\infty$ and a finite subset $A\subseteq\N$, we denote by $\FF_A$ the $\sigma$-algebra of $(X,\mu)$ generated by the $G_{A}$-invariant measurable subsets of $X$, i.e., measurable subsets $Y\subseteq X$ such that $\mu(Y\triangle g\cdot Y)=0$ for all $g\in G_A$. 

\begin{definition}
    A p.m.p.\ action $G\curvearrowright(X,\mu)$ of a closed subgroup $G\leq S_\infty$ is \defin{dissociated} if for all finite disjoint subsets $A,B\subseteq\N$, the $\sigma$-algebras $\FF_A$ and $\FF_B$ are independent.  
\end{definition}

Notice that dissociation implies ergodicity by taking $A=B=\emptyset$. Dissociation plays a key role in the theory of exchangeable arrays, see for instance \cite[Chap.~7]{Kallenberg}. In the context of p.m.p.\ actions of infinite permutation groups, Jahel and Tsankov proved dissociation for all p.m.p.\ ergodic actions of the class of groups under consideration in this paper:

\begin{theorem}[{\cite[Thm.~3.4]{JT}}]\label{thm.JTimpliesDdF} Let $G\leq S_\infty$ be a closed subgroup. Assume that $G$ is oligomorphic, has no algebraicity and admits weak elimination of imaginaries. Then for any p.m.p.\ action $G\curvearrowright(X,\mu)$, the following are equivalent:
\begin{itemize}
    \item $\mu$ is ergodic,
    \item $\mu$ is dissociated. 
\end{itemize}
\end{theorem}

In an upcoming work, we will prove that such a result holds for other groups (such as the automorphism group of the universal rational Urysohn space $\mathbb{U}_\Q$) that lies outside of the scope of oligomorphic groups. 

\begin{remark}
    A variant of Theorem \ref{thm.JTimpliesDdF} in the context of $S_\infty$-exchangeable arrays appeared before in the literature, see for instance \cite[Lem.~7.35]{Kallenberg}.
\end{remark}

Let us discuss examples of oligomorphic groups that have no algebraicity and admits weak elimination of imaginaries.

\begin{example}\label{ex.DdFgroups}
   All the examples we will present are obtained as automorphism groups of Fraïssé limits. A Fraïssé limit $\mathbb{F}$ is an ultrahomogeneous structure uniquely defined (up to isomorphism) by a Fraïssé class of finite structures. We refer the reader to \cite[\S~7.1]{Hodge} for more details on Fraïssé limits. Here is a non-exhaustive list of some Fraïssé classes whose automorphism groups is oligomorphic, has no algebraicity and admits weak elimination of imaginaries.
    \begin{enumerate}[label=(\roman*)]
        \item\label{item.first} The class of finite sets. Its Fraïssé limit is $\N$ and its automorphism group is $S_\infty$.
        \item The class of finite linear orders. Its Fraïssé limit is $(\Q,<)$ and its automorphism group is the group $\Aut(\Q,<)$ of order-preserving permutations of $\Q$. 
        \item The class of finite cyclic orders. Its Fraïssé limit is $(\Q/\Z,<)$ and its automorphism group is the group $\Aut(\Q/\Z,<)$ of bijections of $\Q/\Z$ which preserves the dense cyclic order. 
        \item\label{item.poset} The class of partially ordered finite sets. 
        
        \item\label{item.Rado} The class of finite simple graphs. Its Fraïssé limit is the Rado graph $R$. 
        \item The class of $K_n$-free finite simple graphs for some $n\geq 3$. 
        \item The class of $k$-uniform finite hypergraphs for some $k\geq 2$.
        \item The class of $k$-uniform finite hypergraphs omitting a complete $k$-uniform hypergraph for some $k\geq 2$. 
        \item  The class of finite tournaments. A tournament is a directed graph where there is an oriented edge between any two distinct vertices. 
        \item\label{item.Henson} The class of directed finite graphs omitting a (possibly infinite) set of finite tournaments. Henson observed \cite{Henson} that there are continuum many Fraïssé limits obtained this way. 
        \item\label{item.last} The class of finite metric spaces with distance set $\{0,1,\dots,n\}$ for some fixed $n\geq 1$. For $n=1$, we recover \ref{item.first} and for $n=2$, we recover \ref{item.Rado}. 
    \end{enumerate}
    We now explain briefly why the automorphism groups of these structures satisfy the assumptions of Theorem \ref{thm.JTimpliesDdF}. 
    \begin{itemize}
        \item In all the examples \ref{item.first} - \ref{item.last}, the automorphism group of the Fraïssé limit is oligomorphic because there are only finitely many isomorphism types of finite structures generated by $n$ elements in the corresponding Fraïssé class.
        \item In all the examples \ref{item.first} - \ref{item.last}, the automorphism group of the Fraïssé limit has no algebraicity because the corresponding Fraïssé class has the strong amalgamation property, see \cite[(2.15)]{Cameron} for a definition.
        \item A closed subgroup $G\leq S_\infty$ has the \emph{strong small index property} if every subgroup $H\leq G$ of index $<2^{\aleph_0}$ lies between the pointwise and the setwise stabilizer of a finite set $A\subseteq\N$. This property implies weak elimination of imaginaries and has been verified for the examples \ref{item.first} and \ref{item.Rado}-\ref{item.Henson}, see \cite{PaoliniShelahSSIP}. The fact that the other examples weakly eliminate imaginaries is rather standard. \end{itemize} 
    Let us close this example by mentioning that the automorphism groups in \ref{item.first}-\ref{item.last} are pairwise non (abstractly) isomorphic, implying in particular that there are continuum many group satisfying the hypothesis of Theorem \ref{thmintro.pmprigide}. We refer to \cite{PaoliniShelah} for the problem of reconstructing a countable structure from its automorphism group.
\end{example}

\section{Preliminaries}\label{sec.Prelim}

\subsection{Structures and logic actions}
A relational language $\LL=(R_i)_{i\in I}$ is a countable collection of relation symbols, each of which has a given arity $r_i$. A structure $\M$ in the language $\LL$ (with domain $\N$) is a collection of subsets $R_i^{\M}\subseteq \N^{r_i}$ for each $i\in I$, which are interpretations of the abstract relation symbols $R_i$. 

Examples of structures include simple graphs, where the language is a single binary relation, whose interpretation is the set of edges of the graph. Similarly, $k$-hypergraphs are structures, the language being a single $k$-ary relation. We can also see linear (as well as partial) orders as structures with a single binary relation.

We denote by 
\[\Struc_\LL\coloneqq \prod_{i\in I}\{0,1\}^{\N^{r_i}}\]
the space of structures in the language $\LL$. This is a compact space with the product topology and there is a natural continuous $S_\infty$-action on it called the \defin{logic action}: for  $g\in S_\infty$ and $\M$ a structure, $g\cdot\M$ is the structure $\mathbf{N}$ defined by  
\[\forall i\in I, (R_i^\mathbf{N}(x_1,\dots,x_{r_i})=1 \Leftrightarrow R_i^\M(g\inv(x_1,\dots,x_{r_i}))=1).\]
The automorphism group of $\M\in\Struc_\LL$ is the stabilizer of $\M$ for the logic action, i.e.,
\[\Aut(\M)\coloneqq\{g\in S_\infty\colon g\cdot\M=\M\}.\]

Let $G$ be a closed subgroup of $S_\infty$. For all $n\geq 1$, let $J_n$ be the set of orbits of the diagonal action $G\curvearrowright\N^n$ and let $J=\bigcup_{n\geq 1}J_n$. We denote by $\LL_G\coloneqq (R_j)_{j\in J}$ the \defin{canonical language} associated with $G$, where $R_j$ is of arity $n$ for all $j\in J_n$. The \defin{canonical structure} associated with $G$ is the structure $\M_G\coloneqq (R_j^G)_{j\in J}$, where $R_j^G=j\subseteq \N^n$ for all $j\in J_n$. It is easy to check that $\Aut(\M_G)=G$. In this article we will deal with structures which expand the canonical structure associated with a closed subgroup $G\leq S_\infty$. For a language $\LL$ which contains $\LL_G$, an element $\M\in\Struc_{\LL}$ is an expansion of $\M_G$ if the structure $\M_{\upharpoonright\LL_G}\in\Struc_{\LL_G}$ (called the $\LL_G$-reduct of $\M$) obtained from $\M$ by removing the relation symbols from $\LL\setminus\LL_G$ is \textit{equal} to $\M_G$. 

\begin{definition}
Let $G\leq S_\infty$ be a closed subgroup and let $\LL$ be a language which contains $\LL_G$. A \defin{$G$-structure} in the language $\LL$ is an element of $\Struc_{\LL}$ which is an expansion of the canonical structure $\M_G$. 
\end{definition}

 We denote by $\Struc_\LL^G$ the space of $G$-structures in the language $\LL$. This is a closed subset of $\Struc_{\LL}$ and there is a natural continuous $G$-action on $\Struc_\LL^G$ which is induced from the $G$-action on $\Struc_\LL$ and is called the \defin{relativized logic action} \cite[\S2.7]{BeckerKechris}. With this terminology, for any language $\LL$, there is a canonical homeomorphism between $\Struc_\LL$ and $\Struc^{S_\infty}_{\LL\sqcup\LL_{S_{\infty}}}$ and we will therefore consider elements of $\Struc_\LL$ as $S_\infty$-structures. We will use the relativized logic action in an essential way in Section \ref{sec.rigidity} through a universality theorem of Becker and Kechris \cite[Thm.~2.7.4]{BeckerKechris}, which states that whenever $\LL\setminus\LL_G$ contains relations of arbitrarily high arity, then for any Borel action $G\curvearrowright X$ on a standard Borel space, there exists a Borel $G$-equivariant injective map $X\to\Struc_\LL^G$. Let us state the following variation of the latter theorem, which will allows us to translate results about invariant random expansions to results about p.m.p.\ actions.
 
\begin{lemma}[{\cite[Thm.~2.7.4]{BeckerKechris}}]\label{lem.dictionnary}
     Let $G\curvearrowright (X,\mu)$ be a p.m.p.\ action. Let $\LL$ be a language which contains $\LL_G$ and such that $\LL\setminus\LL_G$ contains relations of arbitrarily high arity. Then there exists a $G$-invariant probability measure $\nu$ on $\Struc_\LL^G$ and a p.m.p.\ isomorphism $\pi : (X,\mu)\overset{\simeq}{\longrightarrow}(\Struc_\LL^G,\nu)$ that intertwines the $G$-action on $X$ and the relativiezd logic $G$-action on $\Struc_\LL^G$. In particular, for $\mu$-a.e.\ $x\in X$, $\Stab(x)=\Aut(\pi(x))$. 
 \end{lemma}

\subsection{First-order and infinitary logics}

In this section, we fix a countable relational language $\LL=(R_i)_{i\in I}$. An atomic formula in the language $\LL$ is an expression of the form $R_i(v_1,\dots,v_{r_i})$ for some $i\in I$, where $v_1,\dots,v_{r_i}$ are free variables and $r_i$ is the arity of $R_i$. The set of quantifier-free formulas is the smallest set that contains atomic formulas and is closed under negation, finite conjunction and finite disjunction. We denote by $\LL_{\omega_1,\omega}$ the infinitary logic in the language $\LL$, which is the smallest set containing atomic formulas and closed under negation, under universal and existential quantification and under conjunction and disjunction of any countable family of formulas with a common finite set of free variables. We denote by $\LL_{\omega,\omega}$ the standard finitary first-order logic in the language $\LL$, which consists in first-order formulas, that is, formulas in $\LL_{\omega_1,\omega}$ with only finitely many disjunctions and conjunctions. A formula in $\LL_{\omega_1,\omega}$ without free variables is called a sentence. For an atomic formula $\phi(v_1,\dots,v_{r_i})=R_i(v_1,\ldots,v_{r_i})$ and $\x=(x_1,\ldots,x_{r_i})\in \N^{r_i}$, we say that $\M\in\Struc_\LL$ satisfies $\phi(\x)$ if $R_i^{\M}(x_1,\ldots,x_{r_i})=1$. The interpretation of any non-atomic formula $\phi(v_1,\dots,v_n)\in\LL_{\omega_1,\omega}$ is defined inductively, the interpretation of each symbol corresponding to its usual use in mathematics. For any tuple $\x\in\N^n$, we write $\M\models\phi(\x)$ whenever $\M$ satisfies $\phi(\x)$. The following lemma is straightforward and will be used many times in the sequel. 

 \begin{lemma}\label{lem.facts}
     Let $\M\in\Struc_\LL$, $\phi(v_1,\dots,v_n)\in\LL_{\omega_1,\omega}$ and $x\in\N^{<\omega}$. 
     \begin{enumerate}[label=(\roman*)]
        \item\label{item.Borelsubset} The set $\{\M\in\Struc_\LL\colon \M\models\phi(\x)\}$ is a Borel subset of $\Struc_\LL$. 
         \item\label{item.modelsinvariant} For all $g\in S_\infty$, we have $\M\models \phi(\x)\Leftrightarrow g\cdot\M\models \phi(g(\x))$.
     \end{enumerate}
 \end{lemma}

 A \defin{fragment} in $\LL_{\omega_1,\omega}$ is a set which contains $\LL_{\omega,\omega}$ and is closed under subformula, finite conjunction and disjunction, negation, universal and existential quantification, and substitution of free variables. Any subset $\Sigma\leq\LL_{\omega_1,\omega}$ is contained in a least fragment denoted by $\langle\Sigma\rangle$, which is countable whenever $\Sigma$ is. By definition, $\langle\emptyset\rangle=\LL_{\omega,\omega}$. Let $F$ be a fragment. Given $n\geq 1$, the set of formulas $\phi(v_1,\dots,v_n)\in F$ with $n$-free variables is a Boolean algebra and we denote by $S^n_F$ its Stone space, which is the space of \defin{$F$-types} in $n$ variables. By Stone duality, $S^n_F$ is a compact Hausdorff totally disconnected space, which is moreover metrizable if and only if $F$ is a countable fragment. A basis of clopen sets of $S^n_F$ is given by the sets $[\phi]\coloneqq \{p\in S^n_F\colon \phi\in p\}$ for $\phi\in F$. Given $\M\in\Struc_\LL$ and $\x\in\N^n$ we denote by $\tp^\M_F(\x)$ the $F$-type in $S^n_F$ defined by 
 \[\tp^\M_F(\x)\coloneqq\{\phi\in F\colon \M\models\phi(\x)\}.\]
 The following lemma is the equivalent of Lemma \ref{lem.facts} for $F$-types. 
 \begin{lemma}
     Let $F$ be a fragment and $\x\in\N^n$. The following holds
     \begin{enumerate}[label=(\roman*)]\label{lem.typefacts}
         \item\label{item.tpBorel} The map $\tp_F(\x) : \Struc_\LL\to S^n_F$ is Borel.
         \item\label{item.tpinvariant} For all $\M\in\Struc_\LL$ and $g\in S_\infty$, we have $\tp_F^\M(\x)=\tp_F^{g\cdot\M}(g(\x))$. 
         \item\label{item.memeorbitememetype} For all $\M\in\Struc_\LL$ and all $\y\in \Aut(\M)\cdot\x$, we have $\tp_F^\M(\x)=\tp_F^\M(\y)$.
     \end{enumerate}
 \end{lemma}
 \begin{proof}
     Let us prove \ref{item.tpBorel}. For all $\phi\in F$, observe that \[\{\M\in\Struc_\LL\colon \tp^\M_F(\x)\in [\phi]\}=\{M\in\Struc_\LL^G\colon \M\models\phi(\x)\},\]
 which is Borel by Lemma \ref{lem.facts} \ref{item.Borelsubset}. Therefore, the map $\tp_F(\x): \Struc_\LL\to S^n_F$ is Borel. The proof of \ref{item.tpinvariant} is a straightforward consequence of Lemma \ref{lem.facts} \ref{item.modelsinvariant}. Finally, \ref{item.memeorbitememetype} is a particular case of \ref{item.tpinvariant}. 
 \end{proof}
 We finally need the notion of quantifier-free types. The space of quantifier-free types in $n$ variables is the Stone space of the Boolean algebra of quantifier-free formulas with $n$ free variables. Given $\M\in\Struc_\LL$ and $\x\in\N^n$, we denote by $\mathrm{qftp}^\M(\x)$ the quantifier-free type defined by
 \[\mathrm{qftp}^\M(\x)\coloneqq\{\phi(v_1,\dots,v_n)\text{ quantifier-free formula}\colon\M\models \phi(\x)\}.\]
 An important remark is that if $\M$ is a $G$-structure for some closed subgroup $G\leq S_\infty$, then for all $\x,\y\in \N^{<\omega}$, $\mathrm{qftp}^{\M}(\x)=\mathrm{qftp}^{\M}(\y)$ implies that $\x$ and $\y$ are in the same $G$-orbit. An immediate corollary is that for any countable fragment $F$, $\tp^{\M}_F(\x)=\tp^{\M}_F(\y)$ implies that $\x$ and $\y$ are in the same $G$-orbit.

 \subsection{Back-and-forth}

Let us fix in this section a closed subgroup $G\leq S_\infty$ and a countable relational language $\LL$ which contains $\LL_G$. The following lemma, which is almost tautological, is at the heart of the technique of back-and-forth that we explain next. 

\begin{lemma}\label{lem:backandforth}
    Let $\M,\mathbf{N}\in\Struc_\LL^G$. Assume that there exist two sequences $A_0\subseteq A_1\subseteq \dots$,  $B_0\subseteq B_1\subseteq \dots$ of finite subsets of $\N$ and a sequence $g_0\subseteq g_1\subseteq\dots$ of bijections $g_n:A_n\to B_n$, which are restrictions of elements in $G$, such that $\bigcup_n A_n=\bigcup_n B_n=\N$ and for all $n\geq 0$, for all $R\in\LL$ of arity $r$, for all $x_1,\dots,x_r\in A_n$, we have 
        \begin{align}\label{conditionétoile}R^\M(x_1,\dots,x_r)=1\Leftrightarrow R^{\mathbf{N}}(g_n(x_1,\dots,x_r))=1.\tag{$*$}\end{align}
    Then $g=\bigcup_ng_n$ belongs to $G$ and satisfies $g\cdot\M=\mathbf N$. If $\M=\mathbf{N}$, then $g\in\Aut(\M)$.
\end{lemma}

The way we will use this lemma in our proofs is by inductively constructing $A_n,B_n$ and $g_n$. Assume for the initial step that we have $A_0,B_0$ with same quantifier-free type and $g_0$ a bijection between $A_0$ and $B_0$ preserving the relations. The strategy is to take $(x_i)_{i\geq 1}$ and $(y_j)_{j\geq 1}$ enumerations of $\N$ and to ensure that $A_n$ contains $x_1,\ldots,x_n$ and $B_n$ contains $y_1,\ldots,y_n$. The key is to construct both sets in order to ensure that for any $x\in \N$ there is $y\in\N$ such that $\mathrm{qftp}^\M(A_n,x)=\mathrm{qftp}^\mathbf{N}(B_n,y)$ and for all $y'\in \N$ there is $x'\in \N$ such that $\mathrm{qftp}^{\mathbf{N}}(B_n,y')=\mathrm{qftp}^\M(A_n,x')$. 

If such is the case, assume that we have constructed $A_n, B_n$ as in the former paragraph and $g_n : A_n\to B_n$ a bijection which is a restriction of an element in $G$. Let $i$ be the smallest index such that $x_i\notin A_n$ and $j$ the smallest index such that $y_j\notin B_n$ (by construction, $i>n$ and $j>n$). By assumption, there exists $y\in\N$ and $x\in\N$
\begin{align}
\label{qftp1}\mathrm{qftp}^\M(A_n,x_i)&=\mathrm{qftp}^\mathbf{N}(B_n,y),\\
\label{qftp2}\mathrm{qftp}^{\mathbf N}(B_n,y,y_j)&=\mathrm{qftp}^\M(A_n,x_i,x).
\end{align}
We then set $A_{n+1}=A_n\cup \{x_i,x\}$, $B_{n+1}=B_n\cup\{y,y_j\}$. Since $x_i\notin A_n$ and $y_j\notin B_n$, the equalities of the quantifier-free types in \eqref{qftp1} and \eqref{qftp2} imply that $x\notin A_n$ and $y\notin B_n$. This allows to extend $g_n$ into a bijection $g_{n+1} : A_{n+1}\to B_{n+1}$ by setting $g_{n+1}(x_i)=y$ and $g_{n+1}(x)=y_j$. Equations \eqref{qftp1} and \eqref{qftp2} and the definition of quantifier-free type ensure that $g_{n+1}$ is indeed the restriction of an element in $G$ and that the condition \eqref{conditionétoile} in Lemma \ref{lem:backandforth} is satisfied. 

\subsection{Non-Archimedean groups with no algebraicity that weakly eliminate imaginaries}

We give a well-known characterization of groups with no algebraicity and admitting weak elimination of imaginaries (see for instance \cite[Lem.~16.17]{Poizat}). We provide here a proof in the language of permutation groups. If $G\leq S_\infty$ is a closed subgroup, we denote by $\Fix(G)$ the set of points in $\N$ which are fixed by $G$.

\begin{lemma}\label{lem.GAGB} 
    Let $G\leq S_\infty$ be a closed subgroup. Then $G$ has no algebraicity and admits weak elimination of imaginaries if and only if $\Fix(G)=\emptyset$ and for all $A,B\subseteq\N$ finite, we have $\langle G_A,G_B\rangle=G_{A\cap B}$. 
\end{lemma}

\begin{proof}
    $(\Rightarrow)$ First, we readily get $\Fix(G)=\emptyset$ because $G$ has no algebraicity. Fix $A,B\subseteq\N$ two finite subsets and let $V=\langle G_A,G_B\rangle$. By definition, $V\leq G_{A\cap B}$. Let us show the reverse inclusion. Since $G$ admits weak elimination of imaginaries, there exists a finite subset $C\subseteq\N$ such that $G_C\leq V$ and $[V:G_C]<+\infty$. We will prove that $C\subseteq A\cap B$. The fact that $G_C$ has finite index in $V$ implies that the $V$-orbit of every $c\in C$ is finite. However, the $G_A$-orbit of every $x\in\N\setminus A$ is infinite since $G$ has no algebraicity. Since $G_A\leq V$, the $V$-orbit of every $x\in\N\setminus A$ is infinite. Therefore $C\subseteq A$. Similarly, $C\subseteq B$. We conclude that $C\subseteq A\cap B$ and thus $G_{A\cap B}\leq G_C\leq V=\langle G_A,G_B\rangle\leq G_{A\cap B}$. 

    \noindent
    $(\Leftarrow)$ Let us first show that $G$ has no algebraicity. Assume that this is not the case. Then there exists a finite subset $A\subseteq\N$ such that $G_A\curvearrowright\N\setminus A$ has a fixed point $b\in\N\setminus A$ (indeed, we know that there exists $A'\subseteq\N$ finite such that $G_{A'}\curvearrowright \N\setminus A'$ has a finite orbit, say $\{b_1,\dots,b_n\}$, then $A=A'\sqcup\{b_1,\dots,b_{n-1}\}$ works). Set $B=\{b\}$. Then we have $G_{A\sqcup B}=G_A$ and $G_{A\cap B}=G$. Therefore, $G=\langle G_A,G_B\rangle=\langle G_{A\sqcup B},G_B\rangle=G_B$. This shows that $B\subseteq\Fix(G)$, which is a contradiction. Thus $G$ has no algebraicity. We now prove that $G$ admits weak elimination of imaginaries. Let $V\leq G$ be an open subgroup. The property ``$\langle G_A,G_B\rangle =G_{A\cap B}$ for all $A,B\subseteq\N$ finite'' implies that there exists a unique finite subset $A_0\subseteq\N$ which is minimal (for inclusion) among all finite subsets $A\subseteq\N$ that satisfy $G_A\leq V$. Let us show that $[V:G_{A_0}]<+\infty$. Observe that for all $g\in G$, $G_{g(A_0)}$ is a subgroup of $gVg\inv$ and $g(A_0)$ is minimal among finite subsets $A\subseteq\N$ that satisfy $G_{g(A_0)}\leq gVg\inv$. This implies that for all $g\in V$, we have $g(A_0)=A_0$ and thus $[V:G_{A_0}]<+\infty$. 
\end{proof}

As a direct corollary, we obtain that the combination of weak elimination of imaginaries and no algebraicity implies primitivity.

\begin{corollary}\label{cor.WEI+noAlgimpliesprimitive}
    Let $G\leq S_\infty$ be a closed subgroup. If $G$ has no algebraicity and admits weak elimination of imaginaries, then $G$ is a primitive subgroup of $\Sym(\Omega)$.
\end{corollary}

\begin{proof}
   We show that for all $a\in\N$, $G_a$ is a maximal subgroup of $G$. Fix $U\leq G$ a subgroup strictly containing $G_a$. Let $g\in G\setminus G_a$ and let $b\coloneqq g(a)$. Then $G_a$ and $gG_ag\inv=G_b$ are contained in $G$. By Lemma \ref{lem.GAGB}, $\langle G_a,G_b\rangle =G$ is contained in $U$. So $U=G$ and $G_a$ is maximal. 
\end{proof}

The following lemma will be useful to characterize proper subgroups of $S_\infty$. A permutation $\tau\in S_\infty$ is a transposition if $\tau$ is a $2$-cycle.

\begin{lemma}\label{lem.transposition}
    Let $G\leq S_\infty$ be a transitive, closed subgroup, which has no algebraicity and is primitive. If $G$ contains a transposition, then $G=S_\infty$.
\end{lemma}

\begin{proof}
    We define an equivalence relation $E$ on $\N$ as follows: \[aEb \Leftrightarrow 
    a=b \text{ or the transposition which exchanges }a\text{ and }b \text{ belongs to }G.\] Notice that $E$ is $G$-invariant: if the transposition $\iota$ which exchanges $a$ and $b$ belongs to $G$, then for all $g\in G$, $g\iota g\inv$ is the transoposition which exchanges $g(a)$ and $g(b)$ and it belongs to $G$. $E$ is not the equality since $G$ contains a transposition. By primitivity, $E=\N\times\N$. This shows that $G=S_\infty$.
\end{proof}

From the previous lemma, we get the following corollary, that we will use later.

\begin{corollary}\label{cor.separationGorbites}
    Let $G\lneq S_\infty$ be a transitive, proper, closed subgroup. If $G$ has no algebraicity and is primitive, then for all $a,b\in\N$ distinct, there exists infinitely many disjoint tuples $\z$, all disjoint from $a$ and $b$, such that $(\z,a)$ and $(\z,b)$ lies in different $G$-orbits. 
\end{corollary}

\begin{proof}
    Assume that there is no such tuple $\z$. Fix an enumeration $z_0,z_1,\dots$ of $\N\setminus\{a,b\}$. Then there exists a sequence $(g_n)_{n\geq 0}$ of elements in $G$ such that $g_n(z_i)=z_i$ for all $i\leq n$ and $g_n(a)=b$. The sequence $(g_n)_{n\geq 0}$ converges to the transposition that exchanges $a$ and $b$. But $G$ is a proper subgroup of $S_\infty$, so $G$ contain no transposition by Lemma \ref{lem.transposition}. This yields a contradiction. So there exists such a tuple $\z$. But $G$ has no algebraicity, so there exists infinitely many disjoint such tuples by Neumann's lemma. 
\end{proof}

\section{Invariant Random Expansions}\label{sec.IRE}

In this section we fix a closed subgroup $G\leq S_\infty$ and a countable relational language $\LL$ which contains the canonical language $\LL_G$. Recall that $\Struc_\LL^G$ denotes the compact space of expansions of the canonical structure $\M_G$, which are the structures $\M\in\Struc_\LL$ whose $\LL_G$-reduct $\M_{\upharpoonright\LL_G}$ is equal to $\M_G$. We now turn our attention to invariant random expansions.

\begin{definition}
    A invariant random expansion of $G$ in the language $\LL$ is a Borel probability measure on $\mathrm{Struc}_\LL^G$ which is invariant under the relativized logic action.
\end{definition}

We denote by $\mathrm{IRE}(G)$ (or $\mathrm{IRE}_{\LL}(G)$ if we want to emphasize the language) the space of invariant random expansions of $G$. We will also use $G$-IRE to refer to an element of $\mathrm{IRE}(G)$. Let us give concrete examples of invariant random expansions.

\begin{example}\label{Example.ExofIRE}\mbox{}
    \begin{enumerate}[label=(\roman*)]
        \item For all $p\in(0,1)$, the random simple graph whose vertex set is $\N$ and the edges are i.i.d.\ with distribution $\mathrm{Ber}(p)$ is an $S_\infty$-IRE. 
        \item Let $\mathrm{LO}(\N)\subseteq\{0,1\}^{\N\times\N}$ be the space of linear orders on $\N$. There is a unique $S_\infty$-invariant probability measure $\mu$ on $\mathrm{LO}(\N)$ and it is defined by its values on the cylinders $\{x_1<\dots<x_n\}$ by $\mu(\{x_1<\dots<x_n\})=1/n!$. This gives an $S_\infty$-IRE.
\item\label{item.2graph} The structure we consider in this case is a particular $3$-hypergraph. A $2$-graph is a 3-hypergraph such that there is an even number of hyperedges between any four vertices. A simple way to produce a $2$-graph is to take a graph, put an hyperedge between three vertices if there is an even number of edges between them, and then remove the edges. Any $2$-graph can be obtained this way and we call a graph producing a given $2$-graph $H$ a graphing of $H$. Consider $\mathbf{N}$ the Fraïssé limit of finite $2$-graphs, which can be obtained by the above construction starting from a Rado graph. Denote by $G$ the automorphism group of $\mathbf{N}$. There is a $G$-IRE concentrated on the space of graphings of $\mathbf N$, see \cite[Chap.~2]{Jahel}. One notable feature is that this IRE doesn't come from an $S_\infty$-IRE.
    \end{enumerate}
\end{example}

\subsection{First properties}

In this section, we state several lemmas about IREs of groups with no algebraicity. Recall that $G$ has no algebraicity if for all finite subset $A\subseteq\N$, we have $\mathrm{acl}(A)=A$. We will not prove the measurability of sets and maps in this section, all of them being straightforward, save for $\M\mapsto \Aut(\M)$. The measurability of this map is the object of Lemma \ref{lem.proprietesStab} \ref{item.stabmesurable}. 

\begin{lemma}\label{lem.fixedptinfinite}
    Assume that $G\leq S_\infty$ has no algebraicity and let $\mu\in\mathrm{IRE}(G)$. Then for $\mu$-a.e.\ $\M\in\Struc_\LL^G$, the set $\Fix(\Aut(\M))$ is either empty or infinite. \end{lemma}
\begin{proof}
    Assume that there exists a finite nonempty subset $A\subseteq\N$ such that \[\mu(\{\M\in\Struc_\LL^G\colon \Fix(\Aut(\M))=A\})>0.\]
    Since $G$ has no algebraicity, there exists infinitely many pairwise disjoints sets $A_n$ in the $G$-orbit of $A$ by Neumann's lemma \cite[Cor.~4.2.2]{Hodge}. By $G$-invariance of $\mu$, the sets $\{\M\in\Struc_\LL^G\colon \Fix(\Aut(\M))=A_n\}$ all have the same measure, which is strictly positive. This is absurd since they are pairwise disjoint. 
\end{proof}

The following lemma shows that if $\mu$ is an ergodic invariant random expansion of $G$ which has no algebraicity, then $\mu$-a.s., $\Aut(\M)$ has no algebraicity, apart from the fixed points $\Fix(\Aut(\M))$. 

\begin{lemma}\label{Lem:NoAlg} Assume that $G$ has no algebraicity and let $\mu\in\mathrm{IRE}(G)$. Then for all tuples $\x\in\N^{<\omega}$, the following holds:
 \begin{enumerate}[label=(\roman*)]      
     \item\label{item:fixptsIRM} for $\mu$-a.e.\ $\M\in\Struc_\LL^G$, we have $\Fix(\Aut(\M)_{\x})=\Fix(\Aut(\M))\cup \x$.
    \item\label{item:orbitsinfiniteor1}  for $\mu$-a.e.\ $\M\in\Struc_\LL^G$, the $\Aut(\M)_{\x}$-orbits on $\N$ are either of size $1$ or infinite.
 \end{enumerate}
\end{lemma}

\begin{proof}
     Let us prove \ref{item:fixptsIRM}. It is clear that $\mu$-a.s., we have $\Fix(\Aut(\M))\cup \x\subseteq \Fix(\Aut(\M)_{\x})$. We now prove the reverse inclusion. Assume by contradiction that there exists $a\in\N\setminus\x$ such that \[\mu(\{ \Fix(\Aut(\M)_{\x})\setminus\Fix(\Aut(\M))\text{ contains }a\})>0.\]
     Therefore, there exists $\y\in\N^{<\omega}$ such that
     \[\mu(\{\Fix(\Aut(\M)_{\x})\text{ contains }a\text{ and }\exists g\in\Aut(\M)\colon g(\x)=\y, g(a)\neq a\}) >0.\]
     Conditionally on the set $\{\Fix(\Aut(\M)_{\x})\text{ contains }a\}$, we have that for $\mu$-a.e.\ $\M$, for all $g,h\in\Aut(\M)$, if $g(\x)=h(\x)=\y$ then $g(a)=h(a)$. If we denote by $\z$ the tuple $(\x,\y,a)$, then for all $b\in\N\setminus\z$, the $\mu$-measure of the set
     \[\{\Fix(\Aut(\M)_{\x})\text{ contains }a\text{ and for any }g\in\Aut(\M)\text{ s.t. }g(\x)=\y, g(a)=b\}\]
     is strictly positive and constant along the $G_{\z}$-orbit of $b$. These orbits are infinite since $G$ has no algebraicity, and the former sets are pairwise disjoint, which is a contradiction. 

     Let us now prove \ref{item:orbitsinfiniteor1}. Assume by contradiction that there exists a tuple $\y=(y_1,\dots,y_n)$ of size $n\geq 2$ such that 
     $\mu(\{\y \text{ is an orbit of }\Aut(\M)_{\x}\})>0$. Let $\z\coloneqq(\x,y_1,\dots,y_{n-1})$. Then $\mu(\{\Fix(\Aut(\M)_{\z})\setminus\Fix(\Aut(\M))\text{ contains }y_n\})>0$, which contradicts \ref{item:fixptsIRM}.
 \end{proof}

The following result is a version of Neumann's lemma for IRE. 

\begin{lemma}\label{Lem:notFix} Assume that $G$ has no algebraicity and  let $\mu\in\mathrm{IRE}(G)$. Let $\x\in\N^{<\omega}$. Then for $\mu$-a.e.\ $\M\in\Struc_\LL^G$, either $\x$ contains an element of $\Fix(\Aut(\M))$ or the $\Aut(\M)$-orbit of $\x$ contains an infinite set of pairwise disjoint tuples.

\end{lemma} 

\begin{proof}
We know by Lemma \ref{Lem:NoAlg} that $\mu$-a.s., the orbits of $\Aut(\M)_{\x}$ are either of size $1$ or infinite. Then Neumann's lemma \cite[Cor.~4.2.2]{Hodge} gives us the desired conclusion. 
\end{proof}

\subsection{IREs without fixed point}

Given a closed subgroup $G\leq S_\infty$, we say that $\mu\in\mathrm{IRE}(G)$ has \defin{no fixed point} if for $\mu$-a.e.\ $\M\in\Struc_\LL^G$, the set $\Fix(\Aut(\M))$ is empty.

\begin{lemma}\label{lem.typespurelyatomic}
   Let $G\leq S_\infty$ be a closed subgroup, which has no algebraicity. Let $\mu\in\mathrm{IRE}(G)$ be dissociated. Let $F$ be a countable fragment. If $\mu$ has no fixed point, then for all $\x\in\N^n$, the probability measure $\tp_F(\x)_*\mu$ on $S^n_F$ is purely atomic. 
\end{lemma}

\begin{proof}
    First of all, observe that by Lemma \ref{lem.typefacts} \ref{item.tpinvariant}, for all $\x\in\N^n$, the random variable $\tp_F(\x) : (\Struc_\LL^G,\mu)\to S^n_F$ is $G_{\x}$-invariant and therefore $\FF_{\x}$-measurable, where $\FF_{\x}$ denotes the $\sigma$-algebra of $G_{\x}$-invariant measurable subsets of $\Struc_\LL^G$. Since $\mu$ is dissociated, we deduce that for all $\x,\y\in\N^n$ disjoint, the random variables $\tp_F(\x)$ and $\tp_F(\y)$ are independent.

    If $\mu$ has no fixed point, then by Lemma \ref{Lem:notFix}, for $\mu$-a.e.\ $\M$, there exists $\y$ disjoint from $\x$ in the $\Aut(\M)$-orbit of $\x$. This implies that $\mu$-a.s., there is $\y\in\N^n$ disjoint from $\x$ such that $\tp_F^\M(\x)=\tp_F^\M(\y)$. Assume that $\tp_F(\x)_*\mu$ is not purely atomic. This means that there exists a measurable subset $X\subseteq\Struc_\LL^G$, $\mu(X)>0$, such that $\tp_F(\x)_*\mu(\cdot\mid X)$ is a diffuse measure. There exists $\y\in\N^n$ (deterministic) disjoint from $\x$ and a measurable subset $X'\subseteq X$ of positive measure such that for all $\M\in X'$, $\tp_F^\M(\x)=\tp_F^\M(\y)$. By independence of $\tp_F(\x)$ and $\tp_F(\y)$ and using the fact that $\tp_F(\x)_*\mu(\cdot\mid X')$ is diffuse, we get that 
    \[0=\mu(\{\M\in X'\colon \tp_F^\M(\x)=\tp_F^\M(\y)\})=\mu(X')>0,\]
    which yields a contradiction. Thus, $\tp_F(\x)_*\mu$ is purely atomic.
\end{proof}

\begin{remark}
    If $\mu$ has fixed points, then one can prove with a similar argument that the probability measure $\tp_F(\x)_*\tilde\mu$ is purely atomic, where $\tilde\mu$ is the conditional measure $\mu(\cdot\mid \{\x\cap\Fix(\Aut(\M))=\emptyset\})$.
\end{remark}

For all $n\geq 0$, we denote by $S^n_F(\mu)$ the countable set of $p\in S^n_F$ for which there exists $\x\in\N^n$ such that $\mu(\{\M\in\Struc_{\mathcal L}^G\colon \tp_F^\M(\x)=p\})>0$. In order to analyze in details IRE with no fixed points, we need the following result, a version for $S_\infty$-IRE of which is contained in \cite[Lem.~4.6]{AFKP}. The proof we present here contains no new argument compared to the proof of the aforementioned result.

\begin{theorem}\label{thm.onepointextension}
    Let $G\leq S_\infty$ be a closed subgroup, which has no algebraicity. Let $\mu\in\mathrm{IRE}(G)$ be dissociated. If $\mu$ has no fixed point, then there exists a countable fragment $F$ such that for all $p\in S^n_F(\mu)$, $q\in S^{n+1}_F(\mu)$ and $(\x,z)\in\N^{n+1}$ satisfying $\mu(\{\M\colon \tp_F^\M(\x)=p \text{ and }\tp_F^\M(\x,z)=q\})>0$,
    the following holds $\mu$-a.s.
    \[\forall \y\in \N^n, (\tp_F^{\M}(\y)=p\Rightarrow \exists z' \colon \tp_F^{\M}(\y,z')=q). \]
\end{theorem}

\begin{proof}
    Let $\omega_1$ be the first uncountable ordinal. We will construct a family of countable fragments $(F_\alpha)_{\alpha<\omega_1}$ depending on $\mu$, indexed by countable ordinals and show that for some ordinal $\alpha<\omega_1$, the countable fragment $F_\alpha$ is as wanted. We define $F_\alpha$ by transfinite induction:
    \begin{align*}
        F_0&=\LL_{\omega,\omega},\\
        F_{\alpha+1} &=\Big\langle F_\alpha, \bigwedge_{\phi\in p} \phi(v_1,\dots,v_n) \text{ for all }p\in S^n_{F_\alpha}(\mu) \text{ and }n\geq 0\Big\rangle,\\
        F_{\beta}&=\bigcup_{\alpha<\beta}F_\alpha \text{ if }\beta\text{ is a limit ordinal}.
    \end{align*}
    For all $n\geq 0$ the set $S^n_F(\mu)$ is countable, therefore $F_\alpha$ is a countable fragment for all ordinal $\alpha$. 
    \begin{lemma}
          There is an ordinal $\alpha $ such that for $\alpha<\beta<\omega_1$, all $\x\in \N^n$ and $r\in S^n_{F_{\beta}}(\mu)$, if $\mu(\{\M\colon \tp_{F_\beta}^\M(\x)=r\})>0$, then 
\begin{align}\label{eq.stabilize}\mu(\{\M\colon \tp_{F_\beta}^\M(\x)=r\})=\mu(\{\M\colon \tp_{F_\alpha}^\M(\x)=s\})
\end{align}
where $s\subseteq r$ denotes the restriction of $r$ to the fragment $F_\alpha$. 
    \end{lemma}

    \begin{proof}
        We reproduce the argument from the proof of \cite[Lem.~4.5]{AFKP}. Fix $\x\in \N^n$. For $\alpha<\omega_1$ an ordinal, let us denote by $\mathrm{Sp}(\alpha)(\x)$ the set of $p\in S^n_{F_\alpha}(\mu)$ such that there exists $\beta>\alpha$ and $q\in S^n_{F_\beta}(\mu)$ whose restriction to the fragment $F_\alpha$ is $p$, satisfying
    \begin{align*}0<\mu(\{\M\colon \tp_{F_\beta}^\M(\x)=q\})<\mu(\{\M\colon \tp_{F_\alpha}^\M(\x)=p\}).
\end{align*}
Assume that for all $\alpha<\omega_1$, $\mathrm{Sp}(\alpha)(\x)$ is non-empty. We construct a sequence $(\alpha_\delta)_{\delta<\omega_1}$ of ordinals and we prove that $r_{\alpha_\delta}(\x)\coloneqq\sup_{p\in\mathrm{Sp}(\alpha)(\x)}\mu(\{\M\colon \tp_{F_{\alpha_\delta}}^\M(\x)=p\})$ is a strictly decreasing sequence of reals of lenght $\omega_1$, which can not exist.

Assume $\delta=\gamma +1 $ and that we have constructed $\alpha_\gamma$. Then $r_{\alpha_{\gamma}}$ is realized by a finite number of types $p_1,\ldots,p_k\in \mathrm{Sp}(\alpha_\gamma)(\x)$. Let us take $\beta>\alpha_\gamma$ such that for all $i\leq k$ there is $q_i\in S^n_{F_\beta}(\mu)$ whose restriction to $F_{\alpha_\gamma}$ is $p_i$ and satisfies
\begin{align*}0<\mu(\{\M\colon \tp_{F_\beta}^\M(\x)=q_i\})<\mu(\{\M\colon \tp_{F_{\alpha_\gamma}}^\M(\x)=p_i\}).
\end{align*}
We set $\alpha_{\delta}=\beta$. We have $r_\beta < r_{\alpha_\gamma}$, indeed if $q$ realizes $r_\beta$, take $p $ its restriction to $F_{\alpha_\gamma}$. If $p\in \{p_1,\ldots,p_k\}$ then we are done by definition of $\beta$ and if not, then $\mu(\{\M\colon \tp_{F_\beta}^\M(\x)=q\})\leq \mu(\{\M\colon \tp_{F_{\alpha_\gamma}}^\M(\x)=p\}) <r_{\alpha_\gamma}$.

If $\delta$ is a limit ordinal, then set $\alpha_\delta = \sup_{\gamma <\delta} \alpha_\gamma$. If $\gamma <\delta$ then $\gamma+1<\delta$ and $r_{\alpha_\gamma}<r_{\alpha_{\gamma+1}}\leq r_\delta$.

This is enough to conclude that there must be $\alpha(\x)$ such that $\mathrm{Sp}(\alpha)(\x)$ is empty. Take $\alpha=\sup_{\x\in \N^{<\omega}} \alpha(\x)$, then $F_\alpha$ is as wanted.
    \end{proof}
   Let us show that $F\coloneqq F_\alpha$ is a countable fragment satisfying the conclusion of the theorem. Fix $p\in S^n_F(\mu)$, $q\in S^{n+1}_F(\mu)$ and $(\x,z)\in\N^{n+1}$ satisfying $\mu(\{\M\colon \tp_{F_\alpha}^\M(\x)=p \text{ and }\tp_{F_\alpha}^\M(\x,z)=q\})>0$. Let us define 
\[\psi (v_1,\dots,v_n)\coloneqq \exists v,\bigwedge_{\phi\in q}\phi(v_1,\dots,v_n,v).\]  Since $q\in S^{n+1}_{F_\alpha}(\mu)$, the formula $ \bigwedge_{\phi\in q}\phi(v_1,\dots,v_{n+1})$ belongs to $F_{\alpha+1}$ Therefore, $\psi\in F_{\alpha+1}$. Since $\mu$ has no fixed, we can fix by Lemma \ref{lem.typespurelyatomic}, a type $r\in S^n_{F_{\alpha+1}}$ which is an atom of the probability measure $\tp_{F_{\alpha+1}}(\x)_*\tilde\mu$, where $\tilde\mu$ is the conditional measure $\mu(\cdot\mid \tp_{F_\alpha}^\M(\x)=p)$. Then $\mu(\{\M\colon\tp_{F_{\alpha+1}}^\M(\x)=r\})>0$ and the restriction of $r$ to the fragment $F_\alpha$ is exactly $p$. By equation \eqref{eq.stabilize}, we therefore get that $\mu(\{\M\colon \tp_{F_\alpha}^\M(\x)=p\})=\mu(\{\M\colon \tp_{F_{\alpha+1}}^\M(\x)=r\})$. Let us prove that $\psi\in r$. For any $\M$ such that $\tp_{F_{\alpha}}^\M(\x)=p$ and $\tp_{F_\alpha}^\M(\x,z)=q$ (which is a set of positive measure by assumption), we have that $\M\models \psi(\x)$. This shows that $\psi\in r$ as otherwise we would have  $\mu(\{\M\colon \tp_{F_\alpha}^\M(\x)=p\})>\mu(\{\M\colon \tp_{F_{\alpha+1}}^\M(\x)=r\})$. To conclude, we obtain that $\mu$-a.s., if $\tp_{F_{\alpha}}^\M(\x)=p$, then there exists $z'\in\N$ such that $\tp_{F_{\alpha}}^\M(\x,z')=q$. The same conclusion holds for any $\y\in\N^n$ by $G$-invariance of the measure $\mu$. 
\end{proof}

We can now prove the main theorem of this section. An invariant random expansion $\mu\in\mathrm{IRE}(G)$ is \defin{concentrated on an orbit} if there exists an orbit $O$ of $G\curvearrowright\Struc_\LL^G$ such that $\mu(O)=1$. This definition is legitimate as orbits of Borel actions are indeed Borel, see \cite[Thm.~15.14]{KechrisCDST1995}. The following lemma will be useful to prove that p.m.p.\ actions are essentially free. 

\begin{lemma}\label{Lem.ConcentratedOnOrbit}
    Let $G$ be a Polish group and $G\actson (X,\mu)$ be a p.m.p ergodic action. If for $\mu \otimes \mu$-a.e.\ $(x,y)\in X\times X$, $x$ and $y$ are in the same orbit, then there is a conull orbit of the action.
\end{lemma}

\begin{proof}
    If we denote by $E$ the set of $(x,y)$ such that $x\in G\cdot y$, we have
 \begin{align*}
    1=\mu\otimes \mu(E)&=\int_{X}\left(\int_{X}\mathds{1}_{G\cdot x}(y)d\mu(y)\right)d\mu(x)\\ 
    &=\int_{X}\mu(G\cdot x)d\mu(x),
\end{align*}
By ergodicity of $\mu$, for a.e.\ $x\in X$, the value of $\mu(G\cdot x)$ is either $0$ or $1$. Therefore, there exists an orbit $O$ of the action $G\curvearrowright X$ such that $\mu(O)=1$
\end{proof}

\begin{theorem}[version for IREs]\label{thm:NoFixIsConcOnOrbitIRE}
    Let $G\leq S_\infty$ be a closed subgroup, which has no algebraicity. Let $\mu\in\mathrm{IRE}(G)$ be dissociated. If $\mu$ has no fixed point, then $\mu$ is concentrated on an orbit. 
\end{theorem}

\begin{proof}
 For $\mu\otimes \mu$-a.e.\ $\M, \mathbf N \in\Struc_\LL^G$, we will construct an  element $g\in G $ such that $g\cdot\M =\mathbf N$, i.e., such that for all $\x\in \N^{<\omega}$, we have
  \[\mathrm{qftp}^\M(\x)=\mathrm{qftp}^{\mathbf{N}}(g(\x)).\]
This will imply that $\mu$ is concentrated on an orbit by Lemma \ref{Lem.ConcentratedOnOrbit}.

 Let us use Theorem \ref{thm.onepointextension} to build $g$ via a back-and-forth. We use the notation of Lemma \ref{lem:backandforth} and its following discussion. Take $(x_i)$ and $(y_j)$ two enumerations of $\N$. Let $F$ be a fragment given by Theorem \ref{thm.onepointextension}.

 We start our back-and-forth by defined $A_0=B_0=\emptyset$. Assume that $A_n$ and $B_n$ have been built, each containing the first $n$ elements of each enumeration respectively. Let us build $A_{n+1}$ and $B_{n+1}$. Let $i$ be the smallest index such that $x_i\notin A_n$ and $j$ the smallest index such that $y_j\notin B_n$ (by construction, $i>n$ and $j>n$). By Lemma \ref{lem.typespurelyatomic} and Theorem \ref{thm.onepointextension}, $\mu\otimes\mu$-a.s.\ there is $y\in \N$ such that 
 \[\tp_F^\M(A_n,x_i)=\tp_F^{\mathbf{N}} (B_n,y),\]
 implying that
  \[\mathrm{qftp}^\M(A_n,x_i)=\mathrm{qftp}^{\mathbf{N}} (B_n,y_i).\]
Similarly, there is $\mu\otimes\mu$-a.s. $x\in\N$ such that 
 \[\tp_F^\M(A_n,x_i,x)=\tp_F^{\mathbf{N}} (B_n,y,y_j).\]
The a.s.\ existence of $x$ is also a consequence of Theorem \ref{thm.onepointextension}. We set $A_{n+1}\coloneqq A_n\cup\{x_i,x\}$, $B_{n+1}\coloneqq B_n\cup \{y,y_j\}$ and $g_{n+1}(x_i)\coloneqq y$, $g_{n+1}(x)\coloneqq y_j$ and $g_{n+1}(a)=g_n(a)$ for any $a\in A_n$. The way we constructed $g_n$ using an enumeration ensures it converges for the pointwise topology to an element of $G$. Moreover, its limit $g$ satisfies a.s.\ 
 \[\mathrm{qftp}^\M(\x)=\mathrm{qftp}^{\mathbf{N}}(g(\x)).\qedhere\]
 \end{proof}

Let us finish this section by giving a translation of Theorem \ref{thm:NoFixIsConcOnOrbitIRE} for p.m.p.\ actions, in line with Lemma \ref{lem.dictionnary}.

\addtocounter{theorem}{-1}

\begin{theorem}[version for p.m.p.\ actions]\label{thm:NoFixIsConcOnOrbitpmp}
    Let $G\leq S_\infty$ be a closed subgroup, which has no algebraicity. Let $G\curvearrowright (X,\mu)$ be dissociated. If $\Stab(x)\curvearrowright\N$ has no fixed point for $\mu$-a.e.\ $x\in X$, then $G\curvearrowright(X,\mu)$ is essentially transitive. 
\end{theorem}

\section{Rigidity for dissociated p.m.p.\ actions}\label{sec.rigidity}

The aim of this section is to give a proof of Theorem \ref{thmintro.pmprigiddissociated}. Let us first recall the notions of essential freeness and essential transitivity for p.m.p.\ actions of Polish groups. The \defin{free part} of a p.m.p.\ action $G\curvearrowright (X,\mu)$ of a Polish group $G$ is the set $\{x\in X\colon \forall g\in G\setminus\{1_G\}, g\cdot x\neq x\}$. This is a $\mu$-measurable $G$-invariant set (see Lemma \ref{lem.proprietesStab} \ref{item.stabmesurable}). We say that $G\curvearrowright (X,\mu)$ is  

\begin{itemize}
    \item \defin{essentially free} if the free part is a conull set,
    \item\defin{essentially transitive} if there exists a conull set on which the action is transitive, 
    \item \defin{properly ergodic} if it is ergodic and every orbit has measure $0$. 
\end{itemize}

\subsection{Essentially free and essentially transitive actions}

We will provide concrete examples both of essentially free and of essentially transitive p.m.p.\ ergodic actions in Lemma \ref{lem.generalizedbernoulli}. Before that, let us prove that if $G$ is Polish non-compact, p.m.p.\ ergodic actions cannot be both. 

\begin{lemma}
    Let $G$ be a Polish non-compact group. If $G\curvearrowright (X,\mu)$ is a p.m.p.\ essentially transitive action, then it is not essentially free. 
\end{lemma}

\begin{proof}
    By contradiction, assume that $G\curvearrowright (X,\mu)$ is essentially free and essentially transitive. Thus, there exists a $G$-invariant conull set $X'\subseteq X$ on which the action is free and transitive. In other words, there exists a Borel probability measure $m$ on $G$ which is invariant by left translation, such that $G\curvearrowright (G,m)$ is measurably isomorphic to $G\curvearrowright (X,\mu)$. This shows that $G$ has a probability Haar measure. Thus $G$ is compact and this finishes the proof. 
\end{proof}

\begin{corollary}\label{cor.essfreeimpliesproperlyergodic}
    Any p.m.p.\ ergodic  essentially free action of a Polish non-compact group is properly ergodic. 
\end{corollary}

We can now provides examples both of essentially free and of essentially transitive p.m.p.\ ergodic actions. 

\begin{lemma}\label{lem.generalizedbernoulli}
    Let $G\lneq S_\infty$ be a proper, closed subgroup, which has no algebraicity and is primitive. Let $(A,\kappa)$ be a standard probability space. Then the p.m.p.\ action $G\curvearrowright (A,\kappa)^{\otimes\N}$ is ergodic. If moreover $(A,\kappa)$ is purely atomic, then $G\curvearrowright (A,\kappa)^{\otimes\N}$ is essentially transitive. Else, $G\curvearrowright (A,\kappa)^{\otimes\N}$ is essentially free. 
\end{lemma}

\begin{proof}
    The ergodicity of $G\curvearrowright (A,\kappa)^{\otimes\N}$ is proved verbatim in \cite[Prop.~2.1]{KechrisTsankov}, the only assumption needed is that the orbits of  $G\curvearrowright\N$ are infinite, which is weaker than having no algebraicity. Assume that $(A,\kappa)$ is purely atomic and let us prove that $G\curvearrowright (A,\kappa)^{\otimes\N}$ is essentially transitive. Let $\mu=\kappa^{\otimes\N}$. In order to prove that $G\curvearrowright (A,\kappa)^{\otimes\N}$ is essentially transitive, we will prove that for $\mu\otimes\mu$-a.e.\ $((a_n)_{n\in\N},(b_n)_{n\in\N})\in A^\N\times A^\N$, there exists a random element $g\in G$ such that $ g\cdot(a_n)_{n\in\N}=(b_n)_{n\in\N}$. For this we run a back-and forth argument.

    Fix $x_1,x_2,\dots$ and $y_1,y_2,\dots$ two deterministic enumerations of $\N$. We will inductively construct $g$, as a limit of partial bijections $g_n : A_n\to B_n$ with $g_0\subseteq g_1\subseteq \dots$, ensuring that, at the $n$-th step, $A_n$ contains $x_1,\dots,x_n$, $B_n$ contains $y_1,\dots,y_n$, and that $g_n$ is the restriction of some element of
    $G$. To initiate the back-and-forth we define $g_0$ on the empty set and define its image as the empty set. Let us now assume that we have defined $g_n$ as wanted with domain $A_n$ containing $x_1,\ldots,x_n$ and image $B_n$ containing $y_1,\ldots,y_n$. Let us define $g_{n+1}$ by defining an image for $x_{n+1}$ and a preimage for $y_{n+1}$. If $x_{n+1}\in A_n$, then $g_n(x_{n+1})$ is already defined. Otherwise, we know that $g_n$ is the restriction of some element $g\in G$. Fix such an element and let $z=g(x_{n+1})$.  Since $z\notin B_n$ and $G$ has no algebraicity, the $G_{B_n}$-orbit of $z$ is infinite. Thus there a.s.\ exists $h\in G_{B_n}$ such that $\tilde{x}_{n+1}\coloneqq h(z)$ satisfies $a_{x_{n+1}}=b_{\Tilde{x}_{n+1}}$. Define $g_{n+1}(x_{n+1})\coloneqq\Tilde{x}_{n+1}$ and $(g_{n+1})_{\upharpoonright A_n}=g_n$, making it the restriction of $hg\in G$. We define in a similar manner $g_{n+1}\inv(y_{n+1})$.
    This construction provides for $\mu\otimes\mu$-a.e.\ couple of sequences $(a_n)_{n\in\N}$ and $(b_n)_{n\in\N}$ an element $g$, which belongs to $G$ as a limit of elements in $G$, such that $g\cdot(a_n)=(b_n)$. This shows that $G\curvearrowright (A,\kappa)^{\otimes\N}$ is essentially transitive by Lemma \ref{Lem.ConcentratedOnOrbit}. 
    
    We now tackle the case when $(A,\kappa)$ is not purely atomic. Let $O$ be a measurable set such that $\kappa(O)>0$ and $\kappa(\{ x\}) = 0$ for all $x\in  O$. For $\mu$-a.e.\ $a\in A$, infinitely many coordinates of $a$ are in $O$. Let us denote by $a_O$ the set of those coordinates. Since a.s.\ for any  $i,j\in a_O$, $a_i\neq a_j$, any $g\in G$ such that $g\cdot a=a$ stabilizes $a_O$ pointwise. However, for any $x,y\in \N$ distinct, by Corollary \ref{cor.separationGorbites}, there must be $\z$ contained in $a_O$ such that $(\z,x)$ and $(\z,y)$ are in different $G$-orbits. Therefore for any $x,y\in\N$ distinct, no $g\in G\setminus\{1_G\}$ satisfying $g\cdot a =a$ can send $x$ to $y$ for any $x,y$ distinct. Therefore $G\curvearrowright(A,\kappa)^{\otimes\N}$ is essentially free.
\end{proof}

\begin{remark}
    This result is a special case of Theorem \ref{thmintro.pmprigiddissociated}, as it is straightforward to show that p.m.p.\ actions $G\curvearrowright(A,\kappa)^{\otimes\N}$ are indeed dissociated. 
\end{remark}

\subsection{The proof of the main theorem}
We are now ready to prove the version for p.m.p.\ actions of Theorem \ref{thmintro.pmprigiddissociatedexpanded}.

\begin{theorem}\label{thm.pmprigiddissociatedexpanded}
    Let $G\lneq S_\infty$ be a transitive, proper, closed permutation group. If $G$ has no algebraicity and is primitive, then for any dissociated p.m.p.\ action $G\curvearrowright(X,\mu)$, the following holds: 
    \begin{itemize}
        \item either $\Stab(x)\curvearrowright\N$ has a fixed point for $\mu$-a.e.\ $x\in X$ and in this case  $G\curvearrowright(X,\mu)$ is essentially free,
        \item or $\Stab(x)\curvearrowright\N$ has no fixed point for $\mu$-a.e.\ $x\in X$ and in this case $G\curvearrowright(X,\mu)$ is essentially transitive.
        \end{itemize}
\end{theorem}

\begin{proof}
Assume that $G\curvearrowright(X,\mu)$ is not essentially transitive. By the version for p.m.p.\ actions of Theorem \ref{thm:NoFixIsConcOnOrbitpmp}, there exists a positive measure set of points $x\in X$ for which $\Stab(x)\curvearrowright\N$ has a fixed point. By ergodicity, $\Stab(x)\curvearrowright\N$ has a fixed point for $\mu$-a.e.\ $x\in X$. 
\begin{claim*}
    For $\mu$-a.e.\ $x\in X$, the group $G_{\Fix(\Stab(x))}$ is trivial. 
 \end{claim*}
 \begin{cproof}
     Let $\pi : X\to\{0,1\}^\N$ be the measurable, $G$-equivariant map given by $\pi(x)=\Fix(\Stab(x))$ and let $\nu=\pi_*\mu$. Since $\mu$ is dissociated, so is $\nu$ (it is straightforward to check that dissociation is preserved under factor). Therefore $\nu$ is a product measure. Fix two distinct points $a,b\in\N$. For $\mu$-a.e.\ $x\in X$, there exists by Corollary \ref{cor.separationGorbites} infinitely many tuples $\z\in\N^{<\omega}$ such that $(\z,a)$ and $(\z,b)$ lies in different $G$-orbits. Since $\nu$ is a product measure, there is for $\mu$-a.e.\ $x\in X$ one such $\z$ which is contained in $\Fix(\Stab(x))$. This shows that for $\mu$-a.e.\ $x\in X$, $a$ and $b$ lies in different $G_{\Fix(\Stab(x))}$-orbits. Therefore $G_{\Fix(\Stab(x))}$ is trivial almost surely. 
 \end{cproof}
 So for $\mu$-a.e.\ $x\in X$, the group $\Stab(x)$, which is a subgroup of $G_{\Fix(\Stab(x))}$, is trivial, which concludes the proof. 
\end{proof}

\begin{remark}\label{rem.Sinftynonessfreenonesstrans}
    The group $S_\infty$ admits p.m.p.\ ergodic actions that are neither essentially free nor essentially transitive. In fact, if $\kappa$ is any probability measure on $[0,1]$ which is neither purely atomic, nor a diffuse measure, then the p.m.p.\ action $S_\infty\curvearrowright([0,1],\kappa)^{\otimes\N}$ is neither essentially free nor essentially transitive. Notice that such actions fail to satisfy the first item of Theorem \ref{thm.pmprigiddissociatedexpanded}, since they have fixed points and are not essentially free. 
\end{remark}

\section{Invariant Random Subgroups of Polish groups}\label{sec.IRS}

\subsection{Definition}
In this section we define the notion of invariant random subgroups for Polish groups. Let $G$ be a Polish group. We denote by $\Sub(G)$ the space of closed subgroups of $G$. The \defin{Effros $\sigma$-algebra} is the $\sigma$-algebra on $\Sub(G)$ generated by the sets
\[\{H\in\Sub(G)\colon H\cap U\neq\emptyset\},\]
where $U$ varies over open subsets of $G$. The following lemma is probably well-known but we were not able to locate a proof in the literature. 
\begin{lemma}
    If $G$ is a Polish group, then $\Sub(G)$ equipped with the Effros $\sigma$-algebra is a standard Borel space.
    \end{lemma}
    
    \begin{proof}
        If $X$ is a standard Borel space and $\mathcal{F}(X)$ denotes the space of closed subsets of $X$, then $\mathcal{F}(X)$ is a standard Borel space when equipped with the $\sigma$-algebra generated by the sets $\{F\in\FF(X)\colon F\cap U\neq\emptyset\}$ where $U$ varies over open subsets of $X$ \cite{Effros}. Therefore, in our case $\FF(G)$ is standard Borel. So it remains to show that $\Sub(G)$ is Borel in $\FF(G)$. By the selection theorem of Kuratowski and Ryll-Nardzewski \cite[Thm.~12.13]{KechrisCDST1995}, there exists a countable sequence of Borel maps $d_i : \FF(G)\to G$ with $i\in I$ such that for all nonempty $F\in \FF(G)$, the set $\{d_i(F)\colon i\in I\}$ is dense in $F$. But a closed subset $F\in \FF(G)$ belongs to $\Sub(G)$ if and only if $1_G\in F$ and $d_i(F)d_j(F)\inv\in F$ for all $i,j\in I$. This implies that $\Sub(G)$ is Borel in $\FF(G)$ and thus $\Sub(G)$ is a standard Borel space. 
    \end{proof}

\begin{remark}
If $G$ is a Polish \textit{locally compact} group, then $\Sub(G)$ is usually endowed with the Chabauty topology, which is the topology generated by the sets
\[\{H\in\Sub(G)\colon H\cap U\neq \emptyset, H\cap K=\emptyset\}\]
where $U$ varies over open subsets of $G$ and $K$ over compact subsets of $G$. In this case, $\Sub(G)$ is a compact Hausdorff space and its Borel $\sigma$-algebra is the Effros $\sigma$-algebra. However, for non locally compact Polish groups, the Chabauty topology is not Hausdorff in general (one can adapt the proof of \cite[Thm.~4.4.12]{Beer} to get that this is indeed not the case for many Polish groups including $S_\infty$). 
\end{remark}

\begin{remark}
    If $d$ is a compatible complete metric on the Polish group $G$, then  the Wijsman topology on $\Sub(G)$ is a Polish topology whose Borel $\sigma$-algebra is the Effros $\sigma$-algebra, see for instance \cite[Sec.~1]{KechrisSpaceActions}. However,  the Wijsman topology depends a priori on the choice of the compatible complete metric $d$ on $G$. 
\end{remark}

The $G$-action by conjugation on $\Sub(G)$ is Borel and we are interested in the probability measures invariant under this action.

\begin{definition} Let $G$ be a Polish group. 
An \defin{Invariant Random Subgroup} of $G$ is a Borel probability measure on $\Sub(G)$ that is invariant by conjugation.
\end{definition}
 
We denote by $\mathrm{IRS}(G)=\mathrm{Prob}(\Sub(G))^G$ the space of invariant random subgroups of $G$. This is a standard Borel space equipped with the $\sigma$-algebra generated by the maps $\mu\mapsto \mu(A)$ with $A$ varying over Borel subsets of $\Sub(G)$ \cite[Thm.~17.23 and 17.24]{KechrisCDST1995}. We say that $\nu\in\IRS(G)$ is \defin{concentrated on a conjugacy class} if there exists an orbit $O$ of the $G$-action by conjugation on $\Sub(G)$ such that $\nu(O)=1$. Recall that orbits of Borel actions are indeed Borel \cite[Thm.~15.14]{KechrisCDST1995} so this definition makes sense.

We now explain how to construct IRSs. Let $G$ be a Polish group and let $G\curvearrowright (X,\mu)$ be a p.m.p.\ action of $G$. Recall that for us, a p.m.p.\ action of a Polish group is a Borel action on some standard Borel space with a Borel invariant probability measure. For $x\in X$, let $\Stab(x)\coloneqq\{g\in G\colon g\cdot x=x\}$ denotes the stabilizer of $x$. We will prove that the law of the stabilizer of a $\mu$-random point is indeed an IRS of $G$. For this, we need the following lemma. 

\begin{lemma}\label{lem.proprietesStab}
    Let $G$ be a Polish group and $G\curvearrowright (X,\mu)$ a p.m.p.\ action. Then
    \begin{enumerate}[label=(\roman*)]
        \item\label{item.stabferme} for all $x\in X$, $\Stab(x)$ is a closed subgroup,
        \item\label{item.stabmesurable} the map $\Stab: x\in X\mapsto \Stab(x)\in\Sub(G)$ is $\mu$-measurable.
    \end{enumerate}
\end{lemma} 
\begin{proof}
By \cite[Thm.~5.2.1]{BeckerKechris}, there exists a Polish topology on $X$ whose Borel $\sigma$-algebra is that of $X$, such that the action $G\curvearrowright X$ is continuous. Let us fix such a topology. The proof of \ref{item.stabferme} is obvious: stabilizers are closed because the action is continuous. Let us prove \ref{item.stabmesurable}. Let $U\subseteq G$ be open and let $B_U\coloneqq \{H\in\Sub(G)\colon H\cap U\neq\varnothing\}$. Let us prove that $\Stab\inv(B_U)$ is analytic (the continuous image of a Borel set in a Polish space). Then \cite[Theorem 21.10]{KechrisCDST1995} will allow us to conclude that $\Stab$ is $\mu$-measurable. First, we have \begin{align*}
      \Stab^{-1}(B_U)&=\{x\in X\colon \exists g\in U, g\cdot x=x\} \\
      &=\pi(B),
\end{align*}where $B=\{(x,g)\in X\times U\colon g\cdot x=x\}$ and $\pi : X\times U\to X$ denotes the projection onto the first coordinate. So we need to prove that $B$ is Borel. But $B$ is the preimage under the continuous map $(x,g)\in X\times U\to (x,g\cdot x)\in X\times X$ of the diagonal set $\{(x,x)\colon x\in X\}$, which is Borel (because $X$ is Polish).
\end{proof}

\begin{remark}
   The fact that stabilizers of Borel actions are closed is due to Varadarajan for locally compact groups, see \cite{Varadarajan} and to Miller \cite[Thm.~2]{Miller} for Polish groups. 
\end{remark}

Therefore, p.m.p.\ actions of Polish groups produce IRSs: if $G\curvearrowright (X,\mu)$ is a p.m.p.\ action of a Polish group, then $\Stab_*\mu$ is an IRS of $G$. In the next theorem, we prove that for closed permutation group groups, every IRS arises this way.

\begin{theorem}\label{thm.IRSrealised}
Let $G$ be a closed subgroup of $S_\infty$ and let $\nu\in\mathrm{IRS}(G)$.  Then there exists a p.m.p.~action $G\curvearrowright (X,\mu)$ such that $\Stab_*\mu=\nu$. 
\end{theorem}

\begin{proof} 
Fix $\nu\in\mathrm{IRS}(G)$ and let us prove that $\nu$ is a stabilizer IRS. Let $X$ be the standard Borel space defined by 
\[X\coloneqq \Sub(G)\times [0,1]^{\N^{<\omega}}.\] Recall that $\N^{<\omega}$ stands for the disjoint union of $\N^n$ for $n\geq 1$. The action of $G$ on $\N^{<\omega}$ induces a Borel action of $G$ on $[0,1]^{\N^{<\omega}}$. Let us now construct a $G$-invariant probability measure $\mu$ on $X$ such that $\Stab_*\mu=\nu$. 

Given $H\in\Sub(G)$, a coloring of $H$ is a map $c : \N^{<\omega}\to[0,1]$ which is constant on each orbit of the action $H\curvearrowright\N^{<\omega}$. 
Let $H\in\Sub(G)$. Then there exists a unique Borel probability measure $\lambda^H$ on $[0,1]^{\N^{<\omega}}$ concentrated on colorings of $H$ such that if $c$ is a random variable with law $\lambda^H$ and $\overline{x_1},\dots,\overline{x_n}\in\N^{<\omega}$ are tuples whose $H$-orbits are pairwise disjoint, then $c(\overline{x_1}), \dots,c(\overline{x_n})$ are i.i.d.\ uniform random variables. By uniqueness, we have that for all $g\in G$, $g_*\lambda^H=\lambda^{gHg\inv}$. Therefore the probability measure on $X$
\[\mu=\int_{X}\lambda^H(c)d\nu(H)\]
is $G$-invariant.

We will now prove that for $\mu$-a.e.\ $(H,c)\in X$, we have $\Stab(H,c)=H$. Remark first that if $g\in\Stab(H,c)$ then $g\in N_G(H)$. The following claim is what we need to conclude that $\Stab_*\mu=\nu$.

\begin{claim*}
Let $H\in\Sub(G)$ and let $g\in N_G(H)$ be such that any orbit of the action $H\curvearrowright\N^{<\omega}$ is invariant by $g$. Then $g\in H$. 
\end{claim*}
\begin{cproof}
Fix an enumeration $x_1,x_2,\dots$ of $\N$. For all $n\geq 1$, the $H$-orbit of $(x_1,\dots,x_n)$ is preserved by $g$. Since $g\in N_G(H)$, this implies that
\[H(x_1,\dots,x_n)=gH(x_1,\dots,x_n)=Hg(x_1,\dots,x_n).\]
Therefore there exists $h_n\in H$ such that $h_n(x_1,\dots,x_n)=g(x_1,\dots,x_n)$. This means that $h_n\to g$ as $n\to +\infty$. Since $H$ is closed, we obtain that $g\in H$. 
\end{cproof}
\end{proof}

\begin{remark}
    Variations of this result for Polish locally compact groups already appeared in the literature. 
Theorem \ref{thm.IRSrealised} was proved for discrete groups by Abért, Glasner and Virag \cite{AbertGlasnerVirag} and for Polish locally compact groups in \cite{ABBGNRS}. 
\end{remark}

\begin{remark}
    Theorem \ref{thm.IRSrealised} is false in full generality for Polish groups, as there exist Polish groups, such as $\Aut(X,\mu)$, which admits no non-trivial p.m.p.\ action, \cite{GlasnerTsirelsonWeiss}, \cite{GlasnerWeissSpatial}. These groups therefore have no essentially free p.m.p.~actions. For such groups $G$, the IRS $\delta_{\{1_G\}}$ is not realized. We thank Anush Tserunyan and Ronnie Chen for pointing us toward such examples. Nonetheless, it would be interesting to understand which Polish groups do realize their IRSs.
\end{remark}

\subsection{From IRS to IRE and vice versa}\label{sec.IRStoIRM}

We will draw a connection between invariant random subgroups and invariant random expansions for closed permutation groups. 

\paragraph{From IRS to IRE.} Fix $G\leq S_\infty$ a closed subgroup. Let us recall the definition of the canonical language. For all $n\geq 1$, let $J_n$ be the set of orbits of the diagonal action $G\curvearrowright\N^n$ and let $J=\bigcup_{n\geq 1}J_n$. The canonical language associated with $G$ is $\LL_G\coloneqq (R_j)_{j\in J}$ where $R_j$ is a relation symbol of arity $n$ for all $j\in J_n$. We define a new language $\LL_{dyn}\coloneqq(T_n)_{n\geq 1}\sqcup\LL_G$ that we call the \defin{dynamical language of $G$}, where $T_n$ is a relation symbol of arity $n$ for each integer $n\geq 1$. 

To any $H\in\Sub(G)$ we associate an expansion $\M_G(H)\in\Struc_{\LL_{dyn}}^G$ of the canonical structure $\M_G$ in the language $\LL_{dyn}$ as follows:
\[\M_G(H)\coloneqq((\mathcal{R}_{H\curvearrowright\N^n})_{n\in\N},(R_j^G)_{j\in J}),\]
where $R_j^G=j\subseteq\N^n$ for all $j\in J_n$ and where $\RR_{H\curvearrowright\N^n}\coloneqq\{(\x,\y)\in\N^n\times\N^n\colon \y\in H\x\}$ is the orbit equivalence relation of the $H$-action on $n$-tuples.

\begin{lemma}\label{lemma.IRStoIRM}
    Let $G\leq S_\infty$ be a closed subgroup. The following hold.
    \begin{enumerate}[label=(\roman*)]
        \item\label{item.MG(H)injective} The map $H\in\Sub(G)\mapsto \M_G(H)\in\Struc_{\LL_{\text{dyn}}}^G$ is Borel, $G$-equivariant and injective.
        \item\label{item.autMG(H)} For all $H\in\Sub(G)$, we have $\Aut(\M_G(H))=N_G(H)$. 
    \end{enumerate}
\end{lemma}

\begin{proof}
    We only show that the map is Borel, the rest being straightforward. For this, it suffices to show that for all $n\geq 1$ and all $\x,\y\in\N^n$, the set $\{H\in\Sub(G)\colon (\x,\y)\in\RR_{H\curvearrowright\N^n}\}$ is Borel. But this set is exactly $\{H\in\Sub(G)\colon H\cap U\neq \emptyset\}$ where $U\subseteq G$ is the open subset consisting of the elements $g\in G$ such that $g(\x)=\y$. This last set belongs to the Effros $\sigma$-algebra, therefore $H\mapsto \M_G(H)$ is Borel.
\end{proof}

To any $\nu\in\IRS(G)$ we can therefore associate an IRE of $G$ as the law of the random expansion $\M_G(H)$ where $H$ is a $\nu$-random closed subgroup of $G$. 

\begin{remark}
This construction is in fact underlying in the proof of Theorem \ref{thm.IRSrealised}, in which we implicitly use this IRE that we expand using colorings of each orbit.
\end{remark}

\paragraph{From IRE to IRS.} Let $G$ be a closed subgroup of $S_\infty$ and fix a countable relational language $\LL$ which contains $\LL_G$. Let $\mu\in\mathrm{IRE}(G)$. If $\M$ is a $\mu$-random expansion, then the law of $\Aut(\M)$ is an IRS of $G$. It may happen that the IRS obtained this way is trivial (that is equal to $\delta_{\{1_G\}}$) whereas the IRE is a nontrivial object of interest. We give details of such an IRE in the next example. 

\begin{example}[The kaleidoscope random graph]
    Let $\LL=(R_n)_{n\in\N}$ be the language consisting in countably many binary relations. Denote by $\PP_2(\N)$ the set of subsets $A\subseteq\N$ with $\lvert A\rvert=2$. Consider the random element $\M\in\Struc_\LL$  obtained by first picking an $S_\infty$-invariant random non-empty subset $A_{\{i,j\}}\subseteq\N$, independently for each $\{i,j\}\in\PP_2(\N)$ and then setting $R_n^\M(i,j)=1$ if and only if $n\in A_{\{i,j\}}$. The law of $\M$ is indeed an IRE of $S_\infty$. This IRE can be thought as the union of countably many random graphs on $\N$, each of which having its edges labeled by a different color. The theory of such an IRE is studied in \cite[Ex.~3.2]{AFKP} and \cite[\S5.1]{AFNP} . This is an example of what they call a properly ergodic structure, which is the main object of study in \cite{AFKP}. If $\mu$ denotes the law of the IRE $\M$, then the p.m.p.~action $S_\infty\curvearrowright (\Struc_\LL,\mu)$ is measurably conjugate to $S_\infty\curvearrowright ([0,1],\mathrm{Leb})^{\otimes\PP_2(\N)}$ which is easily seen to be a properly ergodic p.m.p.~action.
    \end{example}

\subsection{Rigidity of ergodic IRSs of $S_\infty$}

We have already proved in Theorem \ref{thmintro.pmprigiddissociated} that the p.m.p.\ ergodic actions of any proper, transitive, closed subgroup $G\lneq S_\infty$, which is oligomorphic, has no algebraicity and weakly eliminates imaginaries, are either essentially free or essentially transitive. In particular, this implies that for any such group $G$, any ergodic $\nu\in\IRS(G)$ is concentrated on a conjugacy class. On the other hand, we have seen in Remark \ref{rem.Sinftynonessfreenonesstrans} that $S_\infty$ admits p.m.p.\ ergodic actions that are neither essentially free nor essentially transitive.
However, we prove that such a behavior never appears for ergodic IRSs of $S_\infty$. We therefore obtain the following result.

\begin{theorem}\label{thm.IRSrigidity}
    Let $G\leq S_\infty$ be a transitive closed subgroup. If $G$ is oligomorphic, has no algebraicity and admits weak elimination of imaginaries, then any ergodic $\nu\in\mathrm{IRS}(G)$ is concentrated on a conjugacy class. 
\end{theorem}

\begin{proof} As mentioned in the above discussion, if in addition $G$ is a proper subgroup of $S_\infty$, then this is a consequence of Theorem \ref{thmintro.pmprigide}. Therefore, the proof is dedicated to the case $G=S_\infty$. Let $\nu\in\IRS(G)$ be ergodic. By Theorem \ref{thm.JTimpliesDdF}, $\nu$ is dissociated. In order to prove the result, it suffices to prove by Theorem \ref{thm:NoFixIsConcOnOrbitpmp} that for $\nu$-a.e.\ $H\in\Sub(G)$, the action $N_G(H)\curvearrowright\N$ has no fixed point (the normalizer subgroup $N_G(H)$ is indeed the stabilizer of $H\in\Sub(G)$ for the $G$-action by conjugation on $\Sub(G)$). 

\begin{claim*}
    One of the following condition holds:
    \begin{itemize}
        \item either for $\nu$-a.e.\ $H\in\Sub(G)$, the action $H\curvearrowright\N$ has no fixed point,
        \item or for $\nu$-a.e.\ $H\in\Sub(G)$, the action $H\curvearrowright\N$ has infinitely many fixed points. 
    \end{itemize}
\end{claim*}
\begin{cproof}
    The proof is identical to that of Lemma \ref{lem.fixedptinfinite}. Assume that there exists a finite nonempty subset $A\subseteq\N$ such that 
    \[\nu(\{H\in\Sub(G)\colon \Fix(H)=A\})>0.\]
    Since $G$ has no algebraicity, there exists infinitely many pairwise disjoint sets $A_n$ in the $G$-orbit of $A$ by Neumann's lemma \cite[Cor.~4.2.2]{Hodge}. By $G$-invariance of $\nu$, the sets $\{H\in\Sub(G)\colon \Fix(H)=A_n\}$ all have the same measure, which is strictly positive. This is absurd since they are pairwise disjoint. 
\end{cproof}
Now any point $x\in\N$ which is fixed by $N_{G}(H)$ in in fact fixed by $H$. But if for $\nu$-a.e.\ $H\in\Sub(G)$, the action $H\curvearrowright\N$ has infinitely many fixed points, then for $\nu$-a.e.\ $H\in\Sub(G)$, the action $N_{G}(H)\curvearrowright\N$ has no fixed point: indeed, any permutation whose support is contained in the set $\Fix(H)$ commutes with any element of $H$. Therefore, the symmetric group on the set $\Fix(H)$ is contained in $N_G(H)$ (this last argument is the only place where $G=S_\infty$ is used). Therefore, $\nu$ is concentrated on a conjugacy class.
\end{proof}

\section{Further discussions}\label{sec.discussion}

\paragraph{On the existence of essentially transitive IREs and p.m.p\ actions.}In their seminal paper \cite{AFP},  Ackerman, Freer and Patel proved that for any countable relational language $\LL$ and any $\M\in\Struc_\LL$ with no algebraicity, there exists an IRE of $S_\infty$ supported on the orbit of $\M$. Equivalently, for any closed subgroup $H\leq S_\infty$ with no algebraicity, there exists a $S_\infty$-invariant probability measure on $S_\infty/H$. In another paper with Kwiatkowska \cite{AFKwP}, they moreover characterize the cardinality of the set of such measures. This description gives us a great understanding of essentially transitive p.m.p.\ actions of $S_\infty$. Regarding these works, a natural question is: 
\begin{question}[version for IREs]
    Given a closed subgroup $G\leq S_\infty$, is there a natural condition on an expansion $\mathbf N$ of the canonical structure $\M_G$ associated with a given closed subgroup $G\leq S_\infty$, such that there exists a $G$-IRE concentrated on the $G$-orbit of $\mathbf N$?
\end{question}

\addtocounter{theorem}{-1}

\begin{question}[version for p.m.p.\ actions]
    Given a closed subgroup $G\leq S_\infty$, is there a natural condition on a closed subgroup $H\leq G$ such that there exists $G$-invariant probability measure on $G/H$?
\end{question}

This question has been answered in some special cases in \cite{Ackerman} and \cite{AFNP}.
The following observations may help to answer this question, however they also suggest to us that this problem might be difficult.

\begin{enumerate}[label=\arabic*)]
    \item For any closed $G\leq S_\infty$ and any $S_\infty$-IRE $\mu$ in a language $\LL$ (disjoint from $\LL_G$), one readily gets a $G$-IRE in the language $\LL\sqcup\LL_G$. Indeed, take $\M$ a random structure with law $\mu$ and define a $G$-IRE $\mathbf N$ by

    \[R^{\mathbf{N}}(\x)\Leftrightarrow \left\{ \begin{array}{c}
        R^{\mathbf{M}}(\x)  \text{ whenever } R\in \LL. \\
         R^{\mathbf{M}_G}(\x)  \text{ whenever } R\in \LL_G. 
    \end{array}\right.\]

    \item There are groups that admit IREs not produced as in 1). This is the case for the IRE of the automorphism group of the Fraïssé limit of 2-graphs considered in Example \ref{Example.ExofIRE} \ref{item.2graph}.

    \item There are expansions which orbits can not be the support of an IRE. We give two examples.
    \begin{enumerate}[label=\roman*)]
        \item[i)] Any expansion of the generic poset into a linear order. The existence of such an IRE would contradict the non-amenability of the automorphism groups of the generic poset.
        \item[ii)] The expansion $\mathbf N$ of $(\Q,<)$ given by adding a unary relation $R$ and for a fixed irrational $\alpha$, setting for all $q\in \Q$, $R^{\mathbf{N}}(q)$ if and only if $q<\alpha$. If the orbit of this expansion was the support of an ergodic $\Aut(\Q,<)$-IRE, we would get a non product $\Aut(\Q,<)$-invariant ergodic measure on $\{0,1\}^\Q$ which does not exist by \cite{JT}.
    \end{enumerate}
\end{enumerate}

 To conclude, let us note that essentially free actions of $S_\infty$ have been analyzed in depth from a model theoretic perspective in \cite{AFKP} through a notion that is called properly ergodic structures, but such a study for other groups is lacking.

\paragraph{On (strongly) dissociated p.m.p.\ actions.} Let $G\leq S_\infty$ be a closed subgroup. A p.m.p.\ action $G\curvearrowright(X,\mu)$ is \defin{strongly dissociated} if for all finite subsets $A,B\subseteq\N$, the $\sigma$-algebra $\FF_A$ and $\FF_B$ are independent conditioned on $\FF_{A\cap B}$. The first author and Tsankov proved if $G\leq S_\infty$ is a closed subgroup, which is oligomorphic, has no algebraicity and admits weak elimination of imaginaries, then any p.m.p.\ action of $G$ is strongly dissociated \cite[Thm.~3.4]{JT}. This nice characterization suggests the following question.

\begin{question}
    Is there a natural (model theoretic) condition classifying closed subgroups $G\leq S_\infty$ whose  p.m.p.\ ergodic actions are all (strongly) dissociated?
\end{question}

\begin{remark}
    In an earlier version of this paper, we introduced the name \emph{dynamically de Finetti} for closed subgroups $G\leq S_\infty$, which have no algebraicity and admit weak elimination of imaginaries and whose p.m.p.\ actions are all strongly dissociated. We removed this terminology in the present version of this paper and preferred to express our result using the classical notion of dissociation.
\end{remark}

\printbibliography
\Addresses
\end{document}